\documentclass[a4paper,12pt]{article}
\usepackage{graphicx,amssymb,amstext,amsmath}

\usepackage[latin1]{inputenc}
\usepackage{ tipa }
\usepackage{amsthm}
\usepackage{dsfont}
\usepackage{amsmath}
\usepackage{amsfonts}
\usepackage{amssymb}
\usepackage{mathrsfs}
\usepackage{graphicx}
\usepackage{graphicx,color}
\usepackage[all]{xy}
\usepackage[total={5.2in,9.1in},
top=1.4in, left=1.55in, includefoot]{geometry}

\newtheorem{theorem}{Theorem}
\theoremstyle{plain}
\newtheorem{maintheorem}{Theorem}
\newtheorem{maincorollary}[maintheorem]{Corollary}

\newtheorem{Defi}{Definition}
\newtheorem{R}{Remark}

\newtheorem{T}{Theorem}[section]
\newtheorem{Pro}{Proposition}
\newtheorem{C}{Corollary}

\newtheorem{Le}{Lemma}[section]

\newtheorem{claim}{Claim}

\newcommand{\ds}{\displaystyle}

\newcommand{\sgn}{\operatorname{sgn}}
\newcommand{\re}{\mathbb{R}}

\newcommand{\cR}{\mathcal{R}}
\newcommand{\cL}{\mathcal{L}}
\newcommand{\cM}{\mathcal{M}}
\newcommand{\NN}{\mathbb{N}}

\newcommand{\interior}{\operatorname{int}}

\newcommand{\RR}{\mathbb{R}}
\newcommand{\cU}{\mathcal{U}}
\newcommand{\ZZ}{\mathbb{Z}}
\newcommand{\QQ}{\mathbb{Q}}
\newcommand{\dist}{\operatorname{dist}}
\newcommand{\diam}{\operatorname{diam}}

\title{\bf{Hausdorff Dimension, Lagrange and Markov Dynamical Spectra for Geometric Lorenz Attractors}}
\author{C. G.  Moreira\footnote{Partially supported by CNPq,
  PRONEX-Dyn.Syst.}, M. J. Pacifico\footnote{ Partially supported by CNPq,
  PRONEX-Dyn.Syst., FAPERJ} \\and S.  Roma\~na }

\date{June, 07, 2018}
\begin{document}

\maketitle

\begin{abstract}
\noindent In this paper, we show that  geometric Lorenz attractors have Hausdorff dimension strictly greater than $2$.
We use this result to show that for a ``large" set of real functions the Lagrange and Markov Dynamical spectrum associated to these attractors has persistently non-empty interior. 
\end{abstract}
\section{Introduction}
\noindent In 1963 the meteorologist E. Lorenz published in the Journal of Atmospheric Sciences \cite{Lorenz} an example of a parametrized polynomial system of differential equations

\begin{eqnarray}\label{E1-Int}
\dot{x}&=& a(y-x) \ \ \ \ \ \ \ \ \  \ \ \ \ \ \ \ \ \ \ a=10 \nonumber\\
\dot{y}&=& rx-y-xz \ \ \ \ \ \ \ \ \  \ \ \ \ \  \ r=28\\
\dot{z}&=& xy-bz \ \ \  \ \ \  \ \ \  \ \ \ \ \ \ \ \ \ \ \ \ b={8}/{3} \nonumber
\end{eqnarray}
\noindent as a very simplified model for thermal fluid convection, motivated by an attempt to understand the foundations of weather forecast.
Numerical simulations for an open neighborhood of the chosen parameters suggested that almost all points in phase space tend to a strange attractor, called the \textit{Lorenz attractor}. However Lorenz's equations proved to be very resistant to rigorous mathematical analysis,
\begin{figure}[hbtp]
\centering
\includegraphics[scale=0.3]{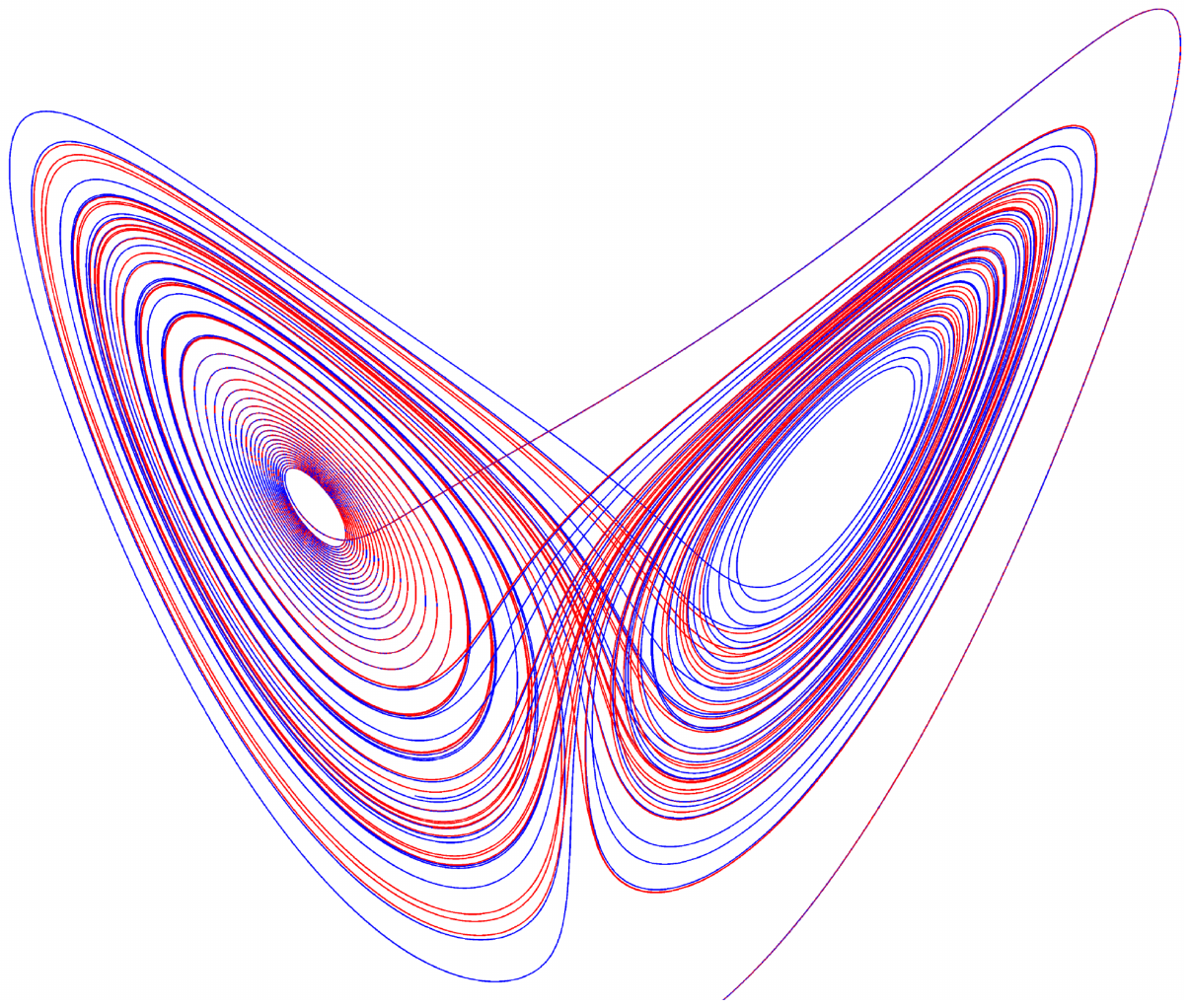}
\caption{Lorenz attractor}
\end{figure}
and also presented very serious difficulties to rigorous numerical study.\\
A very successful approach was taken by Afraimovich, Bykov and Shil'nikov \cite{Afra1}, 
and Guckenheimer, Williams \cite{GW}, independently: they constructed the so called \textit{geometric Lorenz models} for the behavior observed by Lorenz (see section \ref{CGM} for precise definition). These models are flows in $3$-dimensions for which one can rigorously prove 
the coexistence of  an equilibrium point accumulated by regular orbits.
Recall that a regular solution is an orbit where the flow does not vanish. 
Most remarkably, this attractor is robust: it can not be destroyed by  small perturbation of the original flow.
 Taking into account that the divergence of the vector field induced by the
system (\ref{E1-Int}) is negative, it follows that the Lebesgue measure of the Lorenz attractor
is zero. Henceforth, it is natural to ask  about its Hausdorff dimension. 
Numerical experiments give that this value is approximately
equal to 2.062 (cf. \cite{Divakar}) and also, for some parameter, the dimension of the physical invariant measure  lies in the interval $[1.24063,1.24129]$ (cf. \cite{GN}).
In this paper we address the problem to prove that the Hausdorff dimension of a geometric Lorenz attractor is strictly greater that $2$.
In \cite{Afra4} and \cite{Steingenber}, this  dimension is characterized in terms of the pressure of the system and in terms of the Lyapunov exponents and the entropy with respect to a good invariant measure associated to the geometric model. But, in both cases, the authors prove that the Hausdorff dimension is greater or equal than  $2$, but not necessary strictly greater than $2$.
A first attempt to obtain the strict inequality was given in \cite{ML},
where the authors achieve this result in the particular case that  both branches of the unstable manifold of the equilibrium meet the stable manifold of the equilibrium. But this condition is quite strong and extremely unstable.  
One of our goals in this paper is to prove the strict inequality for the Hausdorff dimension for any geometric Lorenz attractor. Thus, our first result is
 \begin{maintheorem}\label{T1}
 The Hausdorff dimension of a geometric Lorenz attractor is strictly greater than $2$.
 \end{maintheorem}
 
To achieve this, since it is well known that the geometric Lorenz attractor is the
suspension of a skew product map with contracting invariant leaves, defined in a cross-section, we start studying the one dimensional map $f$ induced in the space of leaves. 
We are able to prove  the existence of an increasing nested sequence of fat (Hausdorff dimension almost $1$) regular Cantor sets of the one-dimensional map (theorem \ref{L.Principal}). This fact implies that the maximal invariant set $\Lambda_P$ for the skew product (or, to first return map $P$ associated to the flow) has Hausdorff dimension strictly  greater than $1$, and this, on its turn, 
implies that  the Hausdorff dimension of a geometric Lorenz attractor is strictly greater than $2$.
In another words, theorem \ref{T1} is a consequence of the following result
\begin{theorem}\label{L.Principal}
There is an increasing family of regular Cantor sets $C_k$ for $f$ such that 
$$HD(C_k)\to 1 \ \ \text{as} \ \ k\to +\infty.$$ 
\end{theorem}

The proof of this theorem, although non-trivial, is relatively elementary, and combine techniques of several subjects of Mathematics, as Ergodic Theory, Combinatorics and Dynamical Systems (Fractal Geometry).
\vspace{0.2cm}

To announce the next {goal of this paper}, let us recall the classical notions of Lagrange and Markov spectra (see \cite{CF} for further explanation and details).

The {\em{Lagrange spectrum}} $\mathcal{L}$ is a classical  subset of the extended real line, related to Diophantine approximation. Given an irrational number $\alpha$, the first important result about upper bounds for Diophantine approximations is Dirichlet's approximation theorem, stating 
that for all  $\alpha \in \RR \setminus \QQ$, $|\alpha -\frac{p}{q}|< \frac{1}{q^2}$ has a
infinite number of solutions $\frac{p}{q}\in \QQ$.

Markov and Hurwitz improved this result  by verifying that, for all irrational $\alpha$, the inequality $|\alpha - p/q| < \frac{1}{ \sqrt{5}\cdot q^2 } $ has an infinite number of rational solutions $p/q$ and $\sqrt{5}$ is the best constant that work for all irrational numbers. Indeed, for
 $\alpha=\frac{1+ \sqrt{5}}{2}$, the \textit{gold number}, Markov and  Hurwitz also proved that, for every $\epsilon > 0, |\alpha - \frac{p}{q}|< \frac{1}{(\sqrt{5} + \epsilon).q^2}$ has a finite number of solutions in $\QQ$. Searching for better results for a fixed $\alpha \in \re\setminus \QQ$
we are lead to define
$$
k(\alpha)=\sup\{ k > 0: |\alpha - p/q| < 1/(k\, q^2) \, \mbox{has infinitely many rational solutions }\, p/q\}.
$$
Note that the results by Markov and Hurwitz imply that $k(\alpha) \geq \sqrt{5}$ for all $\alpha\in \RR\setminus \QQ$, and $k(\frac{1+ \sqrt{5}}{2})=\sqrt{5}.$ It can be proved that $k(\alpha) =\infty$ for almost every $\alpha \in \re \setminus\QQ.$ 

We are interested in $\alpha\in \re\setminus \QQ$ such that $k(\alpha)< \infty$ (which form a set of Hausdorff dimension $1$).
\begin{Defi}
The Lagrange spectrum  $\cL$ is the image of the map $k$:\\
 $$\cL=\{ k(\alpha), \alpha \in \re\setminus \QQ,\,\,\mbox{ and }\,\, k(\alpha)<\infty\}.$$
\end{Defi}

In $1921$, Perron gave an alternative 
expression for the map $k$, as below. 
Write $\alpha$ in  continued fractions: $\alpha=[a_0,a_1,a_2,\cdots]$. For each $n\in \NN$, define:

$$
\alpha_n=[a_n,a_{n+1}, a_{n+2}, \cdots] \qquad
\beta_n=[0, a_{n-1}, a_{n-2}, \cdots].$$

Then

\begin{equation}\label{Lagrange}
k(\alpha)= \limsup_{n\to\infty}(\alpha_n + \beta_n).
\end{equation}

For a proof of equation (\ref{Lagrange}) see, for instance, \cite[Proposition 21]{CM}.

Markov proved (\cite{Markov}) that the initial part of the Lagrange spectrum is discrete:
$\cL \cap (-\infty, 3)=
\{ k_1=\sqrt{5} < k_2=2 \sqrt{2} < k_3=\frac{\sqrt{221}}{5}< \cdots\}$
with $k_n\to 3, k_n^2 \in \QQ,$ for all $n$.

In $1947$, Hall proved (\cite{Hall}) that the regular Cantor set $C(4)$ of the real numbers in $[0,1]$ in whose continued fraction only appear coefficients $1, 2, 3,4$ satisfies
$C(4) + C(4) = [\sqrt{2} -1, \, 4 (\sqrt{2} - 1)].$
Using the expression
(\ref{Lagrange}) and this result by Hall 
it follows that $[6, \infty) \subset \cL$. That is, the Lagrange spectrum contains a whole 
half line, nowadays called a {\em{Hall's ray}}.

Here we point out $\Lambda = C(4) \times C(4)$ is a horseshoe for a local diffeomorphism related to the Gauss map, which  has Hausdorff dimension $HD(\Lambda)>1$.  Hall's result says that its image $f(\Lambda)=C(4) + C(4)$ under the projection $f(x,y)=x+y$ contains an interval. This  is a key point to get nonempty interior  in $\cL$. 
In $1975$, Freiman  proved (\cite{Freiman}) some difficult results showing that the arithmetic sum of certain (regular) Cantor sets, related to continued fractions, contain intervals, and used them to determine the precise beginning of Hall's ray (the biggest half-line contained in $\cL$) which is
$$
\frac{2221564096\, +\, 283748\, \sqrt{462}}{491993569} \cong 4, 52782956616\cdots 
$$

Another interesting set related to Diophantine approximations is the classical \textit{Markov spectrum} defined by 
\begin{equation*}\label{SMClassic}
\mathcal{M}=\left\{\left(\inf_{(x,y)\in \mathbb{Z}^{2}\setminus(0,0)}|f(x,y)|\right)^{-1}:f(x,y)=ax^2+bxy+cy^2 \ \text{with} \ b^2-4ac=1\right\}.\end{equation*}

Notably the Lagrange and Markov spectrum have a 
{\em{dynamical interpretation.}} 
Indeed, 
the expression of the map $k(\alpha)$ in terms of the continued fraction 
expression of $\alpha$ given in (\ref{Lagrange}) allows to characterize  the Lagrange and Markov spectrum in terms of a shift map in a proper space.
Let $\Sigma=(\NN^*)^\ZZ$ be the set of bi-infinite sequences of integer numbers and 
consider the shift map
$\sigma:\Sigma\to \Sigma, \quad \sigma((a_n)_n)=(a_{n+1})_n$ and
define
$$
f: \Sigma\to \re, \quad f((a_n)_n)=\alpha_0 + \beta_0,$$
where $\alpha_0=[a_0,a_{1}, a_{2}, \cdots]$ and $\beta_0=[0, a_{-1}, a_{-2}, \cdots]$.\\
The Lagrange and the Markov spectrum are characterized as (cf. \cite{CF} for more details)

$$\cL=\{\limsup_k f(\sigma^k((a_n)_n), \, (a_n)_n \in \Sigma\}, \quad
\cM=\{\sup_k f(\sigma^k((a_n)_n), \, (a_n)_n \in \Sigma\}.
$$

\noindent These characterizations lead naturally to a natural extension of these concepts to the context of  dynamical systems.\\
For our purposes, let's consider a more general definition of the Lagrange and Markov spectra. Let $M$ be a smooth manifold, $T=\mathbb{Z}$ or $\mathbb{R}$, and $\phi=(\phi^t)_{t\in T}$ be a discrete-time ($T=\mathbb{Z}$) or continuous-time ($T=\mathbb{R}$) smooth dynamical system on $M$, that is, $\phi^t:M\to M$ are smooth diffeomorphisms, $\phi^0=\textrm{id}$, and $\phi^t\circ\phi^s=\phi^{t+s}$ for all $t,s\in T$.   

Given a compact invariant subset $\Lambda\subset M$ and a function $f:M\to\mathbb{R}$, we define the \emph{dynamical Markov, resp. Lagrange, spectrum} $M(\phi, \Lambda, f)$, resp. $L(\phi, \Lambda, f)$ as 
$$M({\phi, \Lambda, f})=\{m_{\phi, f}(x): x\in\Lambda\}, \quad \textrm{resp.} \quad L({\phi, \Lambda, f})=\{\ell_{\phi, f}(x): x\in\Lambda\}$$
where  
$$m_{\phi, f}(x):=\sup\limits_{t\in T} f(\phi^t(x)), \quad \textrm{resp.} \quad \ell_{\phi, f}(x):=\limsup\limits_{t\to+\infty} f(\phi^t(x)).$$
It can be proved that $L({\phi, \Lambda, f})\subset M({\phi, \Lambda, f})$ (cf. \cite{RM}). In the discrete case, we refer to \cite{RM}, where it was proved that for typical hyperbolic dynamics (with Hausdorff dimension greater than $1$), the Lagrange and Markov dynamical spectra have non-empty interior for typical functions. \\
\indent Moreira and Roma\~na also proved that Markov and Lagrange dynamical spectra associated to generic Anosov flows (including generic geodesic flows of surfaces of negative curvature) typically have nonempty interior (see \cite{RM2} and \cite{R} for more details).

Now we are ready to state our next result. Let $X_0$ be the vector field that defines a geometric Lorenz attractor $\Lambda$ and $U$ an open neighborhood of $\Lambda$ where $X_0$ is defined. 

\begin{maintheorem}\label{T2}
Let $\Lambda$ be the geometric Lorenz attractor associated to  ${X}_0^{t}$. Then
arbitrarily close to ${X^t_0}$, there are a flow ${X}^{t}$ and a neighborhood $\mathcal{W}$ of ${X}^{t}$ such that, if $\Lambda_{Y}$ denotes the geometric Lorenz attractor associated to $Y\in \mathcal{W}$, there is an open and dense set $\mathcal{H}_{Y}\subset C^{1}(U,\re)$ such that for all $f\in \mathcal{H}_{Y}$, we have 
\begin{equation*}
\interior(L(Y,\Lambda_Y, f))\neq \emptyset, \quad \interior(M(Y,\Lambda_Y, f))\neq \emptyset
\end{equation*}
where $int \, (A)$ denotes the interior of $A$.
\end{maintheorem}

\subsection{Organization of the text}
This paper is organized as follows. In Section \ref{CGM}, we describe informally the construction of a geometric Lorenz attractor and announce
 the main proprieties used in the text.
 In Section \ref{FCSf} we prove the first main result in this paper, theorem \ref{L.Principal}
 and its consequences, Corollary \ref{Dim. of bi-dim attractor} and  theorem  \ref{T1}.
 In Section \ref{sec-spectra} we proof our last result, theorem \ref{T2}.

\section{Preliminary results: geometrical Lorenz Model}\label{CGM}
In this section we present informally the construction of
 the geometric Lorenz attractor, following 
  \cite{GP,AP}, where the interested reader can find a detailed exposition of this construction.
 
 Let $(\dot x, \dot y, \dot
z)=(\lambda_1 x,\lambda_2 y, \lambda_3 z)$ be a vector field in the cube $[-1,1]^3$, with a singularity at the origin $(0,0,0)$. 
Suppose the eigenvalues $\lambda_i$, $1\le i\le 3,$ satisfy the relations
\begin{equation}
 \label{eq:eigenvalues}
0<-\lambda_3 < \lambda_1<-\lambda_2,\qquad 0<\alpha=-\frac{\lambda_3}{\lambda_1}<1<\beta=-\frac{\lambda_2}{\lambda_1}\,.
\end{equation}

 Consider $S=\{(x,y,1): |x|\leq 1/2, |y|\leq 1/2\}$ and
  $S^{-}=\{(x,y,1) \in S:  x < 0\}, S^{+}=\{(x,y,1) \in S :  x >0\}$ and $S^\star=S\setminus \Gamma$,
  with $\Gamma=\{(x,y,1)\in S : x=0\}$.
 
  Assume that $S$ is a transverse section to the flow so that every trajectory eventually crosses $S$ in the direction of the negative $z$ axis as in Fig. \ref{nearorigin}. Consider also  $\widetilde{\Sigma}^\pm=\{(x,y,z) : x=\pm 1\}$ and put 
 $\Sigma:=\widetilde{\Sigma}^- \cup \widetilde{\Sigma}^+=\{(x,y,z): |x|=1\}$. 
   For each $(x_0,y_0,1) \in S^\star$ the time $\tau$ such that $X^\tau(x_0,y_0,1) \in \Sigma$ is given by $\tau(x_0)=
  -\frac{1}{\lambda_1}\log(|x_0|)$, which depends on $x_0 \in S^\star$ only and is such that $\tau(x_0) \to + \infty$ when $x_0\to 0$.
Hence we get (where $\sgn(x)= \frac{x}{|x|}$ for $x \neq 0)$  
 $$ X^\tau(x_0,y_0,1)=(\sgn(x_0), y_0 e^{\lambda_2 \tau(x_0)}, e^{\lambda_3 \tau(x_0)})=
 (\sgn(x_0), y_0 |x_0|^{-\frac{\lambda_2}{\lambda_1}}, 
 |x_0|^{-\frac{\lambda_3}{\lambda_1}}).
 $$ 
Let $L: S^\star \to \Sigma$ be given by
\begin{equation}\label{eq:cross_sec}
{L}(x,y,1)=(\sgn(x),yx^\beta,x^\alpha)
\end{equation}
It is easy to see that $L(S^\pm)$ has the shape of a triangle without the vertex $(\pm 1,0,0)$, which are cusps points of the boundary of each of these sets.
 From now on we denote by $\Sigma^\pm$ the closure of $L(S^\star)$. Note that each line segment $S^\star \cap \{x=x_0\}$ is taken to
another line segment $\Sigma \cap \{z=z_0\}$ as sketched in Fig. \ref{nearorigin}.
 Outside the cube, to imitate the random turns of a regular orbit around the origin and obtain
a butterfly shape for our flow,  we let the flow return to the cross section $S$ through a
flow described by a suitable composition of a rotation $R_\pm$, an expansion $E_{\pm\theta}$ and a translation $T_\pm$.
Note that these transformations take line segments 
$\Sigma^\pm \cap \{z=z_0\}$ into line segments $S\cap \{x=x_1\}$ as shown in Figure \ref{nearorigin}, and so does
the composition $T_\pm \circ E_{\pm\theta} \circ R_\pm.$ This composition of linear maps describes a vector field $Y$ in a region outside $[-1,1]^3$, such that the time one map of the associated flow realizes $T_\pm \circ E_{\pm\theta} \circ R_\pm$ as a map
$\Sigma^\pm \to S$.
We note that the flow on the attractor we are constructing will pass though the region between $\Sigma^\pm$ and $S$ in a relatively small time with respect the linearized region.
\begin{figure}\label{nearorigin}
\centering
\includegraphics[scale=0.4]{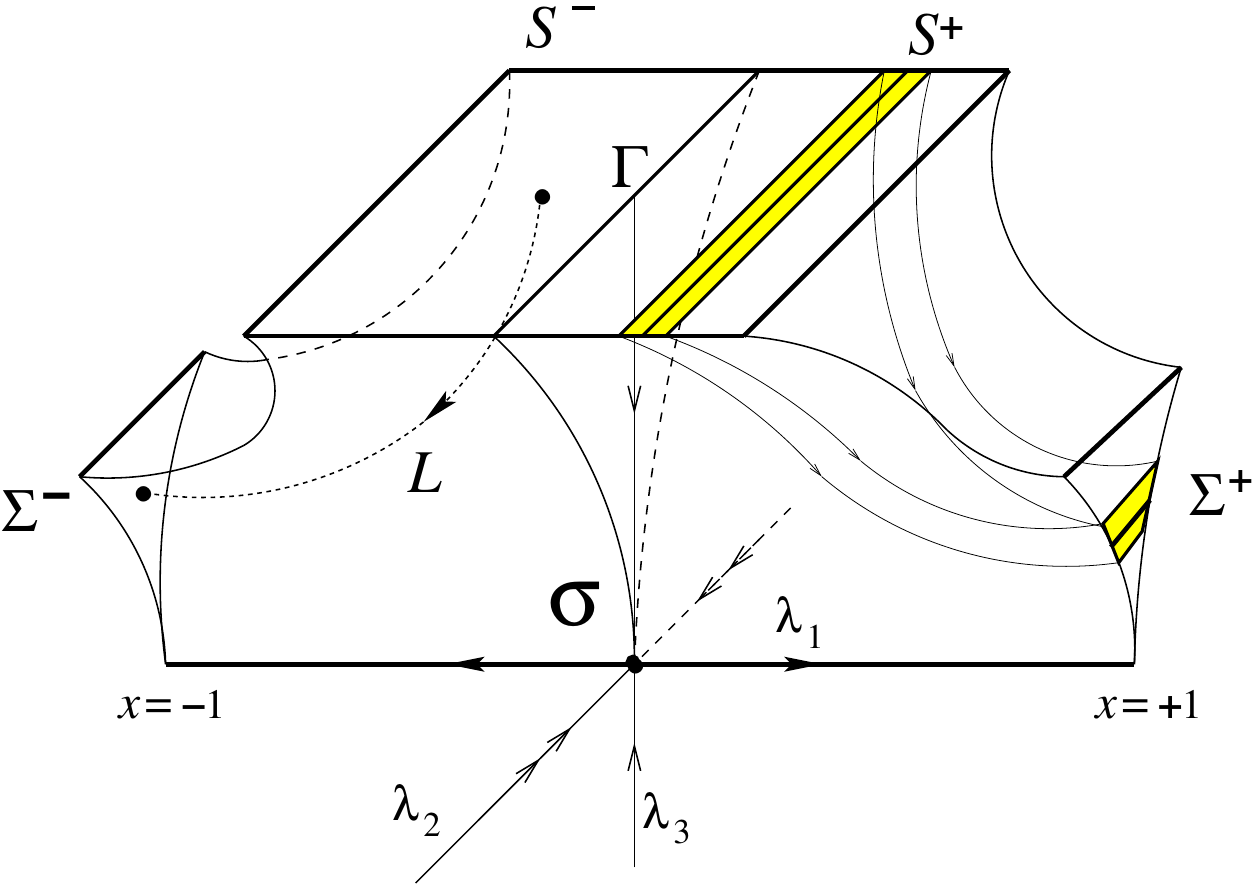}
\hspace{0,5cm}
\includegraphics[scale=0.4]{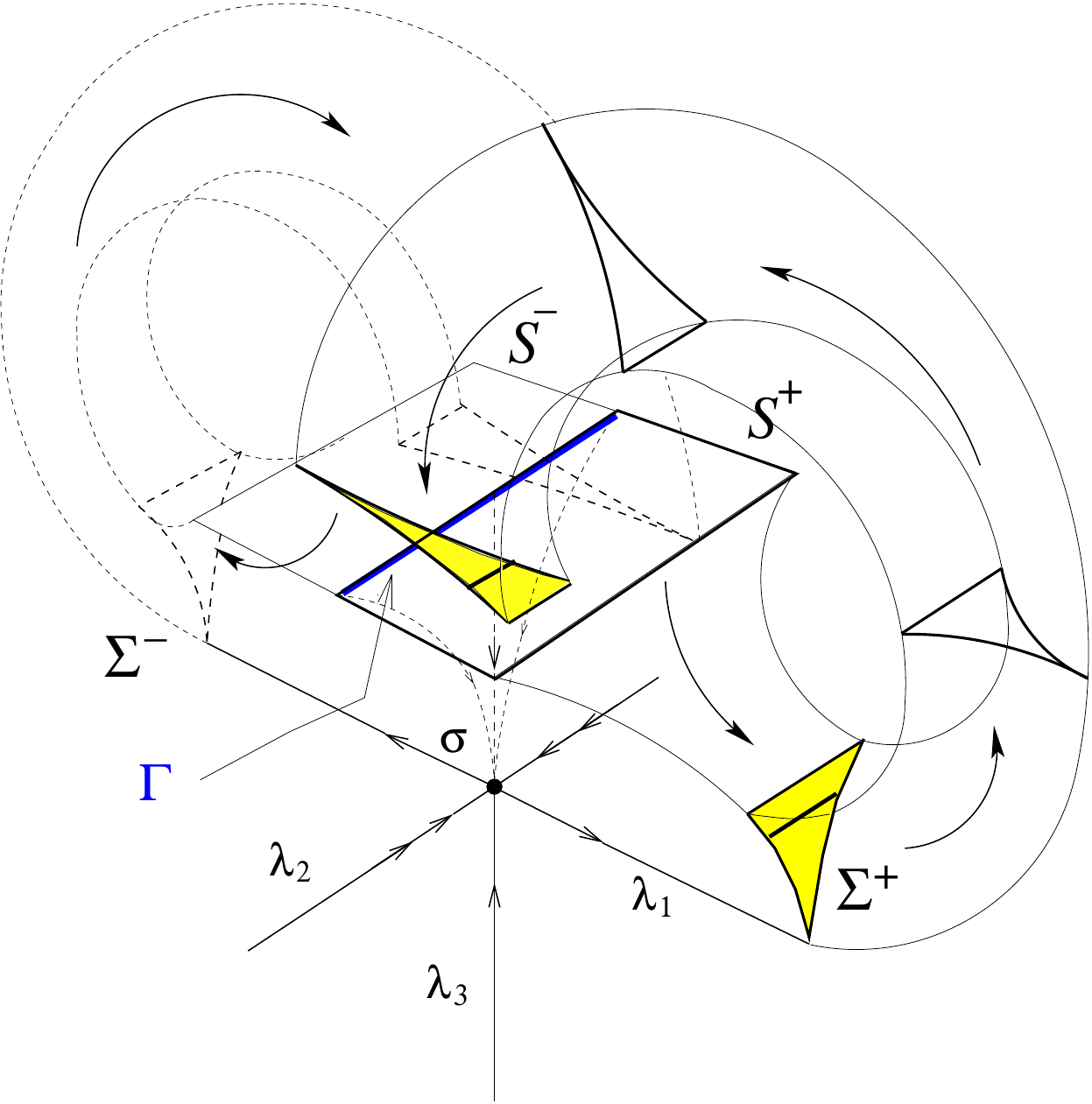}
\caption{Behavior near the origin}
\end{figure}
The above construction enable us to describe for $t\in \RR^+$, the orbit $X^t(x)$ for all $x\in S$: the orbit starts following the linear flow $L$ until $\Sigma^\pm$ and then it will follow $Y$ coming back to $S$ and so on. Now observe that $\Gamma = \{(x,y,1 \in S : x=0\} \subset W^s((0,0,0))$ and so
 the orbit of all $x \in \Gamma$ converges to $(0,0,0)$. 
Let us denote by $W = \{X^{t}(x)\colon x\in S \, ; \, t\in \re^+\}$ the set where this flow acts. The geometric Lorenz flow is the couple $(W, X^{t})$ and 
the geometric Lorenz attractor is the set

\begin{equation}\label{atrator}\Lambda=\overline{\bigcap_{t\geq 0}X^{t}(\Lambda_P)}, \,\, \text{where} \,\,
\Lambda_{P}=\overline{\bigcap_{i\geq 1}P^{i}(S^\star)},
\end{equation}
onde $P:S^\star \to S$ is the Poincar\'e map.

Composing the expression in \eqref{eq:cross_sec} with $R_\pm$,$E_{\pm\theta}$ and $T_\pm$
and taking into account that points in $\Gamma$ are contained in $W^s((0,0,0))$, we  can write an explicit formula for the Poincar\'e map $P$ by 
\begin{equation*}\label{e:Rovelladimension2}
P(x,y)=(f(x),g(x,y))
\end{equation*}
$$f(x)=\left\{\begin{array}{l}
f_0(x^{\alpha}) \ \ \text{if} \ \ x>0\\
f_1(x^{\alpha}) \ \ \text{if} \ \ x<0
\end{array}\right.; \ \ \text{with} \ \ f_i=(-1)^{i}\theta\cdot x +b_i \ \ i=1,2;$$
\noindent and
$$g(x,y)=\left\{\begin{array}{l}
g_{0}(x^{\alpha},y\cdot x^{\beta}) \ \ \text{if} \ \ x>0\\
g_{1}(x^{\alpha},y\cdot x^{\beta}) \ \ \text{if} \ \ x<0
\end{array}\right. .$$

\noindent where $g_1\colon {L_1\times I} \to I$ and $g_{0}\colon {L_2\times I} \to I$ are suitable affine maps, with $L_{1}= [-1/2, 0)$ and $L_{2}=(0,1/2]$.

\subsection{Properties of the one-dimensional map $f$.}\label{ProOneMap} Here we specify the properties of the one-dimensional map $f$ described above:
\begin{itemize}
\item[(f1)] $f$ is discontinuous at $x=0$ with lateral limits $f(0^-)=\frac{1}{2}$ and $f(0^+) = -\frac{1}{2}$,
\item[(f2)] $f$ is differentiable on $I\setminus\{0\}$ and $f'(x) > \sqrt{2}$, where $I=[-1/2,1/2]$,
\item[(f3)] the lateral limits of $f'$  at $x=0$ are $f'(0^{-}) = +\infty=$ and $f' (0^{+} ) =+\infty$.
\end{itemize}

\begin{figure}[hbtp] \label{fig1}
\centering
\includegraphics[scale=0.33]{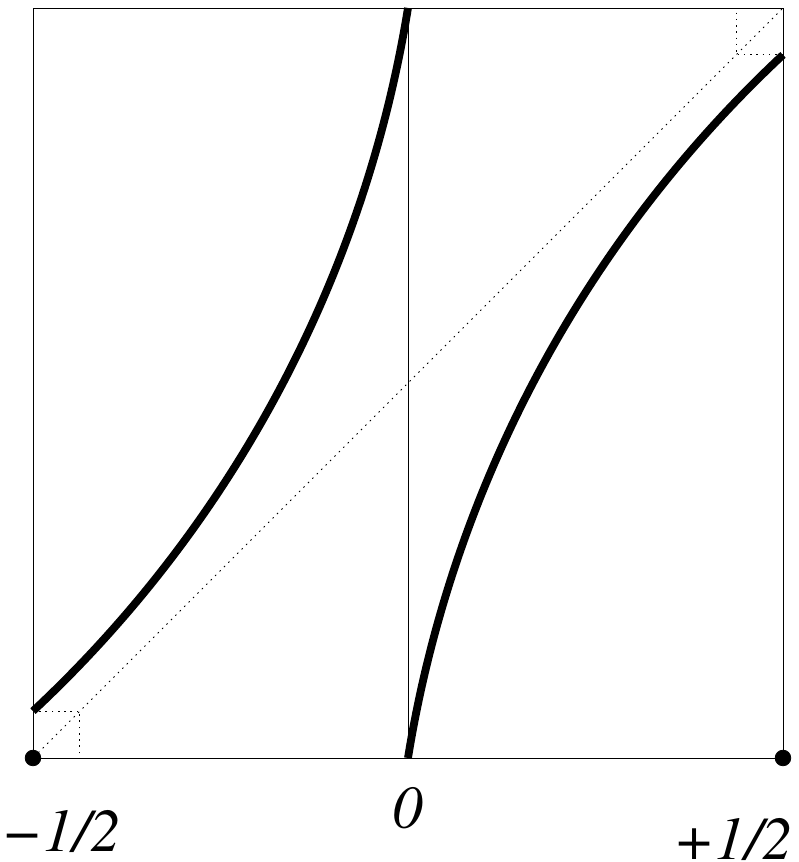} \hspace{0.8cm}
\includegraphics[scale=0.15]{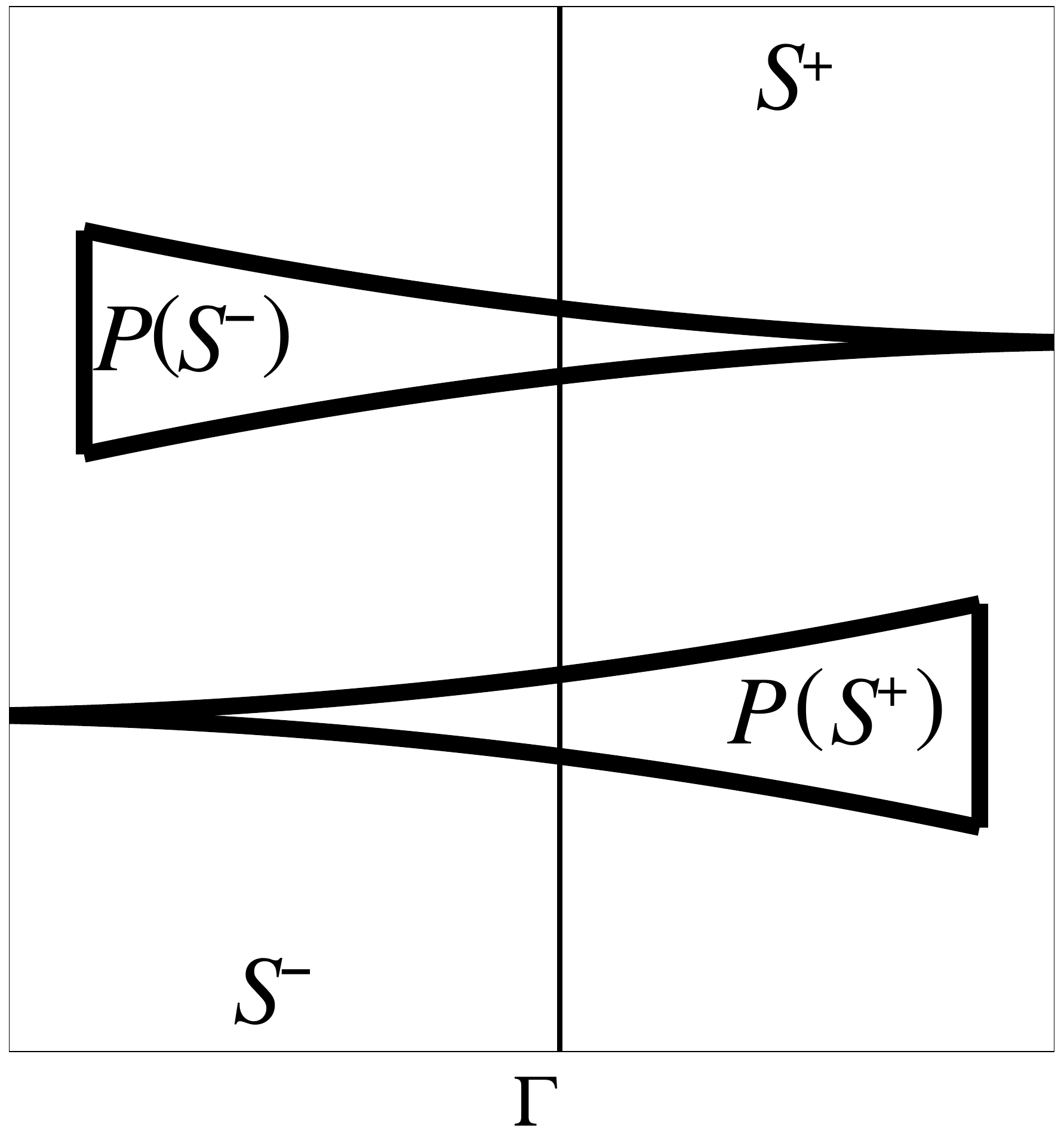}
\caption{The  $1$-dimensional Lorenz map $f$ and the image $P(S^*)$}
\end{figure}

The properties $(f1)$ - $(f3)$ above imply another  important features for the map $f$, as it is shown by the lemma \ref{LEO} below. We will present the proof of R. Willians (cf. \cite[Proposition 1]{Willians}) for the  lemma \ref{LEO}, which we will use  to construct ``almost locally eventually onto" avoiding the singularity $0$ of $f$ (cf. section  \ref{Sec-ALEO}). 
\begin{Le}\label{LEO} Put $I=[-\frac{1}{2},\frac{1}{2}]$. If $J\subset I $ is a subinterval, then there is an integer $n$ such that $f^n (J)=I$. That is, f is locally eventually onto.
\end{Le}
\begin{proof}
Let $J_0 = J$, if $0\notin J$; otherwise let $J_0$ be the bigger of the two intervals $0$
splits $J$ into. Similarly, for each $i$ such that $J_i$ is defined, set 
$$J_{i+1} = \begin{cases} f(J_i) &\mbox{if } 0\notin f(J_i) \\
\text{bigger of two parts $0$ splits $f(J_i)$ into}, & \mbox{if } 0\in f(J_i). \end{cases}$$

Note that $|f(J_{i+1})|>\eta |J_{i+1}|$, where $\eta = \inf |f'|>\sqrt{2}$ and $|\cdot|$ denotes length.
Thus unless $0$ is in both $f(J_i)$ and $f(J_{i+1})$ we have
$$|J_{i+2}|\geq \frac{\eta^2}{2}|{J_i}|.$$
But as $\eta^2 >2$, this last cannot always hold, say
$$0\in f(J_{n-2}) \ \ \text{and} \ \ 0\in f(J_{n-1}).$$
Then $f(J_{n-1})$  contains $0$ and one end point of $I$, so that $J_n$  is one ``half" of $I$. Note that $f(J_n)$
contains the other half, and finally $f^3(J_n)=I$.

\end{proof}
\noindent The next Lemma gives us  the following ergodic  property for $f$ as above (\cite[Corollary 3.4]{M.Viana}).
\begin{Le}\label{ACM}
Let $f \colon  [-1/2,1/2] \setminus \{0\}\to[-1/2,1/2]$ a $C^2$-function,  satisfying  properties
$(f1)$ - $(f3)$ at section $\ref{ProOneMap}$. Then $f$ has some absolutely  continuous invariant probability measure $($with respect to Lebesgue measure $m$\,$)$. Moreover, if $\mu$ is any such measure then $\mu=\varphi\, m$ where $\varphi$ has bounded variation. 
\end{Le}
\subsection{Properties of the map g}

By  definition $g$ is piecewise $C^{2}$ and  the following bounds on its partial derivatives hold:
\begin{enumerate}
\item[(a)] For all $(x,y)\in S^{\star}$ ($x\neq 0$), we have $|\partial_{y}g(x, y)| = |x|^{\beta}$. As $\beta>1$ and $|x| \leq 1/2$
there is $0 < \lambda < 1$ such that
$$|\partial_{y}g| < \lambda. $$
\item[(b)] For $(x, y)\in S^{\star}$ ($x\neq 0$), we have $\partial_{x}g(x, y) = \beta |x|^{\beta-\alpha}$. Since $\beta>\alpha$ and $|x|\leq  1/2$,
we get $|\partial_{x}g| <\infty$.
\end{enumerate}

We note that from the first item above it follows the uniform contraction of the foliation given by the lines $ S \cap \{x=constant\}.$ 
The foliation is contracting in the following sense: there is a constant $C > 0$ such that, for any given leaf $\gamma$ of
the foliation and for $y_1, \, y_2\, \in \gamma$, then
$$
\dist(P^n(y_1), P^n(y_2)) \leq C\, \lambda^n \, \dist(y_1, y_2)) \quad \mbox{as}
\quad n\to \infty.
$$

We notice that 
the geometric Lorenz attractor constructed above is robust, that is, it persists for all nearby vector fields.
More precisely: there exists a neighborhood $U$ in $\re^3$ containing the attracting set $\Lambda$,
such that for all vector fields Y which are $C^1$-close to $X$ the maximal invariant subset in $U$,
 $\Lambda_Y = \bigcap_{t\geq 0}Y^{t} (U)$, is still a transitive $Y$-invariant set.
This is a consequence of the domination of the contraction along the $y$-direction over the expansion along the $x$-direction (see e.g. \cite[Session 3.3.4]{AP}).
Moreover, for every $Y$ $C^1$-close to $X$, the associated Poincar\'e map preserves a
contracting foliation $\mathcal{F}_Y$ with $C^1$ leaves. 
It can be shown that the holonomies along the leaves are in fact H\"older-$C^1$ (see \cite{AP}). Moreover, if we
have a strong dissipative condition on the equilibrium $O$, that is, if $\beta > \alpha+k$ for some $k\in \mathbb{Z}^{+}$ (see the definitions of $\alpha,\, \beta$ as functions of the eigenvalues of $0$ in (\ref{eq:eigenvalues})), it can be show that $\mathcal{F}_Y$ is  a $C^k$ smooth foliation \cite{SV}, and so the holonomies along the leaves of $\mathcal{F}_Y$ are $C^k$ maps. In particular, for
strongly dissipative Lorenz attractors with $\beta > \alpha+k$ the one-dimensional quotient map is
$C^k$ smooth away from the singularity (cf. \cite{SV}).

We finish this section noting that putting together the observations above and the results proved in \cite[Section 3.3.4]{AP}, we easily deduce the following result:

\begin{Pro}\label{Proprobusto}
There is a neighborhood $\mathcal{U} \supset X$ such that for all $Y\in \mathcal{U}$, 
if $f_Y$ is the quotient map $f_Y \colon S^{\ast}/\mathcal{F}_Y \to S/\mathcal{F}_Y$  
associated to the corresponding Poincar\'e map $P_Y$, then the properties (f1)-(f3) from subsection \ref{ProOneMap} are still valid. Moreover, there are constants $C, C_1 > 1$ uniformly on a $C^2$ neighborhood of $X$ such that if $\alpha(Y)=-\frac{\lambda_3(Y)}{\lambda_{1}(Y)}$ 
is the continuation of $\alpha=- \frac{\lambda_3}{\lambda_1}$ obtained for the initial flow $X^t$ it holds
\begin{equation}\label{E3S2}
 \frac{1}{C}\leq \frac{Df_{Y}(x)}{|x|^{\alpha(Y)-1}}\leq C \ \ {\text{and} \ \  \frac{|Df^{2}_{Y}(x)|}{|x|^{\alpha(Y)-2}}\leq C_1}.
\end{equation}
Furthermore, the condition (f3) ensures that $f_Y$ has enough expansion to easily prove that every $f_Y$ is \textit{locally eventually onto} for all $Y$ close to $X$.
\end{Pro}

\subsection{Almost Locally Eventually Onto}\label{Sec-ALEO}
In this section we shall use an argument similar
to the one given in lemma \ref{LEO}, to achieve a property of $f$
 fundamental  for the construction of the family of Cantor sets 
in Theorem \ref{L.Principal}. Roughly speaking, we shall prove
the existence of a number $a$ arbitrarily close to one, depending only on $f$, such that for any interval $J\subset I$,  we have 
\begin{enumerate}
\item[(1)] an interval $J'\subset J$ such that $0 \notin J'$ and  with size equal to a
fixed proportion of the size of $J$. 
\item[(2)] a number $n=n(J)$ such that the restriction
$f^n|J'$, $f^n: J' \to L_1^a$ is a diffeomorphism, where $L_1^a=[f(1-a),0)$.  Moreover, we obtain a control on the distortion at each step  $f^{j}$, for all 
$1\leq j\leq n-1$.
\end{enumerate}
To do that, we start with an auxiliary result.

\begin{Le}\label{L1'-ALEO}
There is a constant $\kappa>0$ such that for all interval $J\subset I\setminus\{0\}$ such that $0\in f(J)$ and $0\in f^2(J)$, then 
$$|J|\geq \kappa.$$
\end{Le}
\begin{proof}
We denote $0_1\in [-\frac{1}{2},0]$ and $0_2\in[0,\frac{1}{2}]$ the preimage of $0$ in each branch of $f$, that is, $f(0_i)=0,\, i\in \{1,2\}$. Consider also the two preimages of $0_i$ $0_i^1, 0_i^2$, $i=1,2$ in $[-\frac{1}{2},0]$ and $[0,\frac{1}{2}]$, respectively. 
As $0\in f(J)$ and $0\in f^2(J)$, then $0,0_i\in f(J)$ for some $i$, and thus we get  that some of the intervals $J_1=[0_1^1,0_1], \, J_2=[0_1,0_2^1], \, J_3=[0_1^2, 0_2], \ \text{and} \ J_4=[0_2,0_2^2]$ is contained in $J$.  Thus, taking $\kappa=\min \{|J_1|, |J_2|, |J_3|,|J_4|\}$ we finish the proof.
\end{proof}
Recall that $\eta^2>2$. Now we consider a number $0<a<1$ satisfying 
\begin{equation}\label{eq1'}
a^2\eta^{2}>2 \ \ \text{and} \ \ 1-a<\kappa.  
\end{equation}

For an interval $J\subset I$, we will use the number $a$ satisfying equation (\ref{eq1'}) to define an interval 
$\tilde{J}\subset J$  avoiding the singularity
$0$ and  obtained  cutting a small part of  $J$ with length $(1-a)|J|$.
In this direction, we proceed as follows.\\

Given any interval $J=(b,c)\subset I\setminus \{0\}$, we denote by $J_a$ the subinterval of $J$ cutting an interval of size  $(1-a)|J|$  on the  closest side to zero, that is, 
$$J_{a} = \begin{cases} (b, ac+ (1-a)b), &\mbox{if } c<0 \\
(ab+(1-a)c,c), & \mbox{if } b>0. \end{cases}$$
Note that, if $c=0$, then $J_a=(b, (1-a)b)$ and
if $b=0$ $J_a=((1-a)c,c)$.
It is clearly  that $|J_a|=a|J|$ and  $0\notin J_a$. \\

\noindent When an interval $J=(b,0)$ or $J=(0,c)$ has size large enough ($|J|>(1-a)$),  we  define the subinterval  $_aJ$ of $J$ cutting an interval of size $(1-a)$ of $J$ in the side of the point $0$, in other words, 
$$_aJ=(b, -(1-a)) \ \ \text{or} \ \ _aJ=((1-a),c).$$
It is clear that both kind of intervals, $J_a$ and $_aJ$,  avoid the singularity. \\ 

Recall that  $0_1\in (-\frac{1}{2},0)$ and $0_2\in (0,\frac{1}{2})$ are the preimages of $0$, that is, $f(0_i)=0$, $i=1,2$.
For the next lemma assume that $0_1\in J_1=(b,0)$ and $0_2\in J_2=(0,c)$. So, for $a$ sufficiently close to $1$ we have that $0_1\in\,_aJ_1=(b,a-1)$ and  $0_2\in\, _aJ_2=(1-a,c)$. Therefore $0\in f( _aJ_1)$ and $0\in  f(_aJ_2)$. Denote $f(_aJ_1)^+$ and $f(_aJ_2)^+$ the bigger of two parts $0$ splits $f(_aJ_1)$ and  $f(_aJ_2)$ into, respectively.

The next lemma says that if $a$ is sufficiently close to $1$, then it is easy to determine $f(_aJ_1)^+$ and $f(_aJ_2)^+$ explicitly.  

\begin{Le}\label{R2-ALEO}
Keeping the notation of above, if $a$ is  close enough  to $1$, then 
$$f(_aJ_1)^+=f([0_1,a-1]) \ \ \text{and} \ \ f(_aJ_2)^+=f([1-a, 0_2]).$$
\end{Le}
\begin{proof}
Since $0_1 \in \, _aJ_1=(b,a-1)$ and  $0_2\in \, _aJ_2=(1-a,c)$, we need only to prove that 
$$|f([b,0_1])|<|f([0_1, a-1])| \ \text{and} \ |f([0_2,c])|<|f([1-a,0_2])|,$$
for $a$ sufficiently close to $1$. Let's prove the 
left hand  inequality, the other one is analogous. Note that, $f(-\frac{1}{2})\neq -\frac{1}{2}$, then by definition of $f$ we have that $|f([-\frac{1}{2},0_1])|<|f([0_1,0))|$. In particular, for all $b\in [-\frac{1}{2},0_1]$ it holds that $|f([b,0_1])|<|f([0_1,0))|$. Now consider the number $\psi:=|f([0_1,0))|-|f([-\frac{1}{2},0_1])|>0$, which only depend of $f$. 
So, we can taken $a$ sufficiently close to $1$ such that 
$$|f([0_1,0))|-|f([0_1,a-1])|<\dfrac{\psi}{2},$$
and therefore we  conclude that $|f([b,0_1])|<|f([0_1, a-1])|$ as we wished. 
\end{proof}

From now on, we will denote $L^{a}_1:=[f(1-a),0)$.\\

To prove the next lemma, we use the same idea as in the proof of lemma \ref{LEO}, to get  control on the number of iterations required to increase the size of any interval avoiding the singularity in each step. 
\begin{Le}\label{L2-ALEO}
If $J\subset I$ is a subinterval then there are a subinterval  $J'\subset J$  and an integer $n(J)$ such that $f^{n(J)}: J'\to L_1^{a}$ is a  diffeomorphism such that $d(f^{i}(J'),\{0\})>0$, $i=0,\dots,n(J)-1$ and 
$$n(J)\leq 3+\ds\dfrac{\log\dfrac{1}{2|J|}}{\log\dfrac{a^2\eta^2}{2}}.$$
\end{Le}
\begin{proof}
Given $J\subset I$, let $J_0: = J_a$ if $0\notin J$; otherwise let $J_0:=J^{+}_{a}$, where $J^{+}$ is  the biggest connected component of $J\setminus \{0\}$. 
Similarly, for each $i$ such that $J_i$ is defined, set 
\begin{equation}\label{e2L2}
J_{i+1} = \begin{cases} f(J_i)_{a} &\mbox{if } 0\notin f(J_i) \\
f(J_i)^{+}_{a}, & \mbox{if } 0\in f(J_i). \end{cases}
\end{equation}

Note that $|f(J_{i+1})|>\eta |J_{i+1}|$, where $\eta = \inf |f'|>\sqrt{2}$.
Thus, unless $0$ is in both $f(J_i)$ and $f(J_{i+1})$ we have
$$|J_{i+2}|\geq \frac{a^2\eta^2}{2}|{J_i}|.$$
But as $a^2\eta^2 >2$, this last inequality cannot always hold. Let  $n$ be the minimum {number}  such that
\begin{equation}\label{eq3-ALEO}
0\in f(J_{n-2}) \ \ \text{and} \ \ 0\in f(J_{n-1}).
\end{equation}
Thus equation (\ref{eq3-ALEO}) implies that $J_{n-2}$ satisfies the hypothesis of lemma \ref{L1'-ALEO}. Therefore $|J_{n-2}|\geq \kappa$, and as $1-a<\kappa$ (see equation (\ref{eq1'})),  we define the interval $_af(J_{n-2})^{+}:=\tilde{J}_{n-1}\subset f(J_{n-2})^{+}$.
To finish the proof of lemma, we have to consider two cases, depending on the relative position of 
$\tilde{J}_{n-1}$ in the connected components of $I\setminus\{0\}$:\\

\textbf{Case 1:} Assume that $\tilde{J}_{n-1}\subset (0,\frac{1}{2}]$. Thus by definition of $_af(J_{n-2})^{+}$ we have that $\tilde{J}_{n-1}=[1-a, b]$ for some $b>0$. Moreover, as $0\in f(J_{n-1})$, then $0\in f^2(J_{n-2})$, therefore as $\tilde{J}_{n-1}\subset (0,\frac{1}{2}]$, {then arguing as in the proof of } lemma \ref{L1'-ALEO}, we get that $[0_1,0_2^1]\subset f(J_{n-1})$ or  $[0_2,0_2^2]\subset f(J_{n-1})$,  consequently  since $f(0_2^i)=0_2$, then  $0_2\in \tilde{J}_{n-1}$, which implies by  lemma \ref{R2-ALEO}  that $|f((1-a,0_2))|>|f((0_2,b))|$ or equivalently $$f(\tilde{J}_{n-1})^{+}=f([1-a,0_2))=[f(1-a),0)=L^a_{1}.$$
In this case, we define the following sequence of intervals $I_{n-2}=f^{-1}[1-a,0_2]\subset J_{n-2}$ and $I_{i}=f^{-1}(I_{i+1})\subset J_{i},\, 0\leq i \leq n-2$.
 Hence, by construction, the interval $J':=I_0\subset J$  satisfies $$f^{i}(J')\subset J_{i-1} = \begin{cases} f(J_{i-2})_{a} &\mbox{if } 0\notin f(J_{i-2}) \\
f(J_{i-2})^{+}_{a}, & \mbox{if } 0\in f(J_{i-2}). \end{cases}$$ 
Therefore we conclude that 
$$d(f^{i}(J'),\{0\})\geq  \begin{cases} (1-a)\cdot |f(J_{i-2})|\geq (1-a)\cdot|J_{i-2}| &\mbox{if } 0\notin f(J_{i-2}) \\
\, \\
(1-a)\cdot|f(J_{i-2})^{+}|\geq \dfrac{1-a}{2}|J_{i-2}|, & \mbox{if } 0\in f(J_{i-2}). \end{cases}$$
So,  taking  $n(J)=n$ we have that $f^{n(J)}\colon J'\to L_{1}^{a}$ is a diffeomorphism and  it is easy to see that  $d(f^{i}(J'),\{0\})>0$, $i=0,\dots,n(J)-1$.  This concludes the proof of Case 1.\\
\ \\
\textbf{Case 2:} Assume that $\tilde{J}_{n-1}\subset [-\frac{1}{2},0)$. Then,  $\tilde{J}_{n-1}=[c,a-1]$. Thus by the same argument of case 1, we have that  $0_1\in \tilde{J}_{n-1}$ and by lemma \ref{R2-ALEO} we have  $|f((0_1, a-1))|>|f((c,0_1))|$ or equivalently $f(\tilde{J}_{n-1})^{+}=f((0_1,a-1])=(0, f(a-1)]$, then,  $\tilde{J}_{n}: =\ _af(\tilde{J}_{n-1})^{+}=[1-a, f(a-1)]$. Note that for $a$ sufficently  close to $1$, $f(a-1)>0_2$ and therefore $0_2\in \tilde{J}_n$. To conclude our arguments, we note that by lemma \ref{R2-ALEO} $$f(\tilde{J}_{n})^{+}=[f(1-a),0)=L^a_1.$$
In this case, we define the following sequence of intervals $I_{n-1}=f^{-1}[1-a,0_2]\subset \tilde{J}_{n-1}$ and $I_{i}=f^{-1}(I_{i+1})\subset J_{i}$, then using a similar argument of case 1, we have that  the interval $J':=I_0\subset J$ satisfies the condition $d(f^{i}(J'),\{0\})>0$, $i=0,\dots,n(J)-1$ for $n(J)=n+1$. The proof of Case 2 is complete. \\
\ \\
To finish the proof of lemma, it is only left  to estimate  $n(J)$. For this, note that in any case,  by construction $|J_{n-2}|\geq \left(\frac{a^2\eta^2}{2} \right)^{n-2}|J|$ and since $|J_{n-2}|\leq \frac{1}{2}$  the  estimative required for $n(J)$   follows immediately.
\end{proof}

\begin{R}\label{R1} Note that we also proved the next estimative
$$d(f^{i}(J'),\{0\})\geq \dfrac{(1-a)}{2}|J_{i-2}|, \,\,i=0,\dots,n(J)-1,$$ where $J_{i-2}$ are  given by (\ref{e2L2}) above. 

\end{R}


The next corollary will be a fundamental tool for the proof of theorem \ref{L.Principal}, more specifically, see  claim 4 in the proof of theorem \ref{L.Principal}.
\begin{C}\label{C1-ALEO}
Let $m_k\in \mathbb{N}$ be a sequence such that $\ds\lim_{k\to \infty}m_k=\infty$. Then if $J_k\subset I$ with\\ $|J_k|\geq \dfrac{1}{3m_k^3}$, we have:
\begin{itemize}
\item[$(a)$]There is a constant $D$ such that $n(J_k)\leq D\log m_k$.
\item[$(b)$]  There are constants $E>0$ and $\xi>0$ such that for each $i=0,1,\dots, n(J_k)-1$ hold that $\ds\sup_{x\in f^{i}(J'_{k})}|f'|=E\cdot m_{k}^{\xi}$, where $J'_{k}$ is as the lemma \ref{L2-ALEO}.
\end{itemize} 
\end{C}
\begin{proof}
\begin{enumerate}
\item[$(a)$]  If $|J_k|\geq \dfrac{1}{3m_k^3}$  lemma \ref{L2-ALEO} implies that
$$n(J_k)\leq \ds\dfrac{\log3-\log 2+3\log m_k+3\log \frac{a^2\eta^2}{2}}{\log \dfrac{a^2\eta^2}{2}}\leq D\log m_k,$$
where $D=5/\log\frac{a^2\eta^2}{2}$ and we finish the proof of Item (a).
\item [$(b)$] Let $J'_{k}$ be the interval  given by the lemma \ref{L2-ALEO}. Then remark  \ref{R1} provides 
\begin{equation*}
d(f^{i}(J_k'),\{0\})\geq \frac{1-a}{2}|(J_k)_{i-2}|,
\end{equation*}
where $(J_k)_{i-2}$ are defined at (\ref{e2L2}). The   construction of $(J_k)_{i-2}$ gives  $|(J_k)_{i-2}|\geq |J_k|\geq \frac{1}{3m_k^3}$. Thus
\begin{equation}\label{e1C1}
d(f^{i}(J_k'),\{0\})\geq \frac{1-a}{2}\cdot \dfrac{1}{3m_k^3}.
\end{equation} 

  The next step is to estimate the derivative of $f\colon f^i(J_k')\to f^{i+1}(J_k')$. For this purpose, we use
the inequalities (\ref{e1C1}) and (\ref{E3S2}) 
 which  provides that 
$$\ds\sup_{x\in f^{i}(J'_{k})}|f'|\leq C\left(\frac{1-a}{6m_k^3}\right)^{\alpha-1}=C\cdot(1-a)^{\alpha-1}\cdot 6^{1-\alpha}m_{k}^{3\cdot(1-\alpha)}.$$
We take $E=C\cdot(1-a)^{\alpha-1}\cdot6^{1-\alpha}$ and $\xi=3\cdot(1-\alpha)$. This concludes the proof. \\

\end{enumerate}
\end{proof}


\section{Fat Cantor sets for $f$ and proof of Theorem \ref{L.Principal}}\label{FCSf}
The main goal in this section is to prove theorem \ref{L.Principal}, that is, that there are infinitely many regular Cantor sets for the one dimensional map  associated to a Geometric Lorenz Attractor, with Hausdorff dimension ($HD$) very close to $1$. 

Before we announce precisely  this result, let us recall the definition of Hausdorff dimension of a Cantor set and  the notion of {\em{regular}} Cantor set. We refer the reader the book \cite[Chapter 4]{PT} for a nice exposition of the main properties of this kind of Cantor sets. We proceed as follows. 

Let $K\subset \RR$ be a Cantor set and $\mathcal{U}= \{U_i\}_{1\leq i\leq n}$
a finite covering of $K$ by open intervals in $\RR$. We define the diameter $\diam(\mathcal{U})$ as the maximum of $\ell(U_i), 1\leq i \leq n$, where $\ell_i:=\ell(U_i)$ denotes the length of $U_i$. Define $H_\alpha(\cU)= \ds\sum_{1\leq i \leq n}\ell_i^\alpha.$ Then the {\em{Hausdorff $\alpha$-measure }} of $K$ is
$$
\ds m_\alpha(K)=\lim_{\epsilon \to 0} \left( \inf_{{\substack{\cU\,\, cover K,\\ \diam(\cU)< \epsilon}}} H_\alpha(\cU)\right).
$$
One  can  show  that  there  is  an  unique  real number,  the {\em{Hausdorff dimension}} of $K$, which we denote by $HD(K)$, such that for $\alpha < HD(K)$, $m_\alpha(K)=\infty$ and for
$\alpha > HD(K)$, $m_\alpha(K)=0$.

\begin{Defi}\label{Regular Cantor set}
A dynamically defined (or regular) Cantor set is a Cantor set $K\subset \RR$, together with
\begin{enumerate}

\item[i)] 
a disjoint compact intervals $I_1,I_2,\dots, I_r$ such that $K\subset I_1\cup I_2\cup \cdots \cup I_r$ and the boundary of each $I_j$ is contained in $K$;
\item[ii)] there is a $C^{1+\alpha}$ expanding map $\psi$ defined in a neighborhood of $I_1\cup I_2\cup\cdots I_r$ such that, for each $j$, $\psi(I_j)$ is the convex hull of a finite union of some of these intervals $I_s$. Moreover, 
$\psi$ satisfies: 
\begin{itemize}
\item for each $1\leq j\leq r$ and $n$
sufficiently big, $\psi^n(K\cap I_j)=K$;
\item $K=\bigcap \psi^{-n}(I_1\cup I_2\cup \cdots I_r)$.
\end{itemize}

\end{enumerate}
We say that $\{I_1,I_2,\dots, I_r\}$
is a Markov partition for $K$ and that $K$ is defined
by $\psi$.
\end{Defi}
A classical example of regular Cantor set  in $\RR$ is the ternary Cantor set $K_{\frac{1}{3}}$ of the elements of $[0,1]$  which can be written in base $3$ using only digits $0$ and $2$.  The set $K_{\frac{1}{3}}$ is  a regular Cantor set, defined by the map
$\psi\colon [0,\frac{1}{3}]\cup [\frac{2}{3},1] \to  \RR$ given by  
$$\psi(x)= \begin{cases}  \ \ 3x, &\mbox{if } x\in  [0,\frac{1}{3}]\\
-3x+3, & \mbox{if } x\in[\frac{2}{3},1]. \end{cases}$$

There is a class of examples of regular Cantor sets, given by a non trivial  basic set $\Lambda$ associated to a $C^2$ diffeomorphism $\varphi: M\to M$ of a $2$-manifold $M$, which appear in the proof of corollary A. Recall that a basic set is a compact hyperbolic invariant transitive set of $\varphi$ which coincides with the maximal invariant set in a neighborhood of it. Nontrivial means that it does not consist of finitely many periodic orbits. These types of regular Cantor sets, roughly speaking,  are given by the  intersections $W^s(x)\cap \Lambda$ and $W^u(x)\cap \Lambda$, where $W^s(x)$ and  $W^u(x)$ are the stable and unstable manifolds  of $x\in \Lambda$. 
We denote by $K^s:=W^s(x)\cap \Lambda$  the \emph{stable Cantor set} and  $K^u:=W^u(x)\cap \Lambda$ the \emph{unstable Cantor set} (cf. \cite[chap 4]{PT} or \cite[Appendix]{RM2}). 

\indent If $\Lambda$ is a basic set associated to $C^2$ diffeomorphism defined in a surface, then it is locally the product of two regular Cantor sets $K^s$ and $K^u$ (cf. \cite[Appendix 2]{PT}).  
We shall use the following properties of a regular Cantor set, whose proofs can be found in \cite{PT}:\\
\begin{Pro}\label{Prop.2}$\cite[Proposition \,  4]{PT}$
The
Hausdorff dimension of a basic set $\Lambda$ satisfies 
\begin{equation*}\label{dimension of produt}
HD(\Lambda)=HD(K^s\times K^u)=HD(K^s)+HD(K^u).
\end{equation*}
\end{Pro}

\begin{Pro}\label{Prop.3} $\cite[Proposition \,  7]{PT}$\/
If $K$ is a regular Cantor set then
$$0< HD(K)< 1.$$
\end{Pro}

\begin{proof}[\textbf{Proof of Theorem \ref{L.Principal}}]

We construct the Cantor sets inductively. 
Denote $L_1=[-1/2,0)$ and $L_2=(0,1/2]$, and pick  any interval $I_1\subset L_1$. Lemma \ref{LEO} implies that there is an iterate $f^{n^1}$ of $f$ such that $f^{n^1}\colon I_1\to L_1$ is a diffeomorphism. Let $\{J_1^{1},J_2^{1}\}$ be the complementary intervals in $L_1$ of $I_1$. Again  lemma \ref{LEO} implies that there are $I_1^1\subset J_1^1$, $I_2^1\subset J_2^1$, $n^1_1$ and $n_2^1$ such that 
$f^{n^1_i}\colon I_i^1\to L_1$ is a diffeomorphism.

Let  $\{J_1^{11},J_1^{12}\}$ be the complementary intervals of $I_1^{1}$ in $J_1^{1}$ and $\{J_2^{11},J_2^{12}\}$ be the complementary intervals  $I_2^{1}$ in $J_2^{1}$.

Continuing with this process, in the $k$-th step, we obtain $r_k=2^k-1$ intervals $I_1,\dots I_{r_k}$ such that, for each $i\in \{1,\cdots,k\}$, there is $n_i$  so that $f^{n_i}\colon I_{r_i}\to L_1$ is a diffeomorphism. 

Now, let  $\{J_1^{(k)},\dots J^{(k)}_{r_k+1}\}$ be the complementary intervals of $\ds\bigcup_{i=1}^{r_k}I_i$ in $L_1$ and $\mu$ be an the invariant measure given by the lemma \ref{ACM}, which is absolutely continuous w.r.t. Lebesgue and thus, there is a constant $c$ such that  
\begin{equation}\label{E0L3}
\mu(I)\leq c\, m(I)=c\,|I|,
\end{equation}
for any interval $I$.  
Take $\epsilon_{k}=\frac{1}{c}\ds\min_{i}\{\mu(J_{i}^{(k)})\}\leq \ds\min_{i}\{|J_{i}^{(k)}|\}$ and put $m_k=\lfloor\frac{1}{\epsilon_k}\rfloor$ the integer part of $\frac{1}{\epsilon_k}$, that is, $m_k\leq \frac{1}{\epsilon_k}<m_k+1$.\\
Next, split each intervals $J_{i}^{(k)}$ in $2^{m_k}$ intervals $\{J^{(k)}_{i,j}\colon j=1,\dots, 2^{m_k}\}$ pairwise disjoint of equal $\mu$-size. Then, for  $j=1,\dots, 2^{m_k}$, we have
\begin{equation}\label{E1L3}
\frac{1}{2^{m_k}}\geq \frac{|J_{i}^{(k)}|}{2^{m_k}}=|J_{i,j}^{(k)}|\geq\frac{1}{c}\mu(J_{i,j}^{(k)})=\frac{1}{c}\frac{\mu(J_i^{(k)})}{2^{m_k}}\geq \frac{\epsilon_k}{2^{m_k}}>\frac{1}{2^{m_k}(m_k+1)}. 
\end{equation}
Consider the interval $\left(-\frac{1}{m_k^{3}}, \frac{1}{m_k^{3}}\right)$. Since $\mu$ is $f$-invariant, inequality  (\ref{E0L3}) implies that   
\begin{eqnarray}\label{E2L3}
\mu\left(\ds \bigcup_{j=1}^{4m_k}f^{-j}\left(-\frac{1}{m_k^{3}}, \frac{1}{m_k^{3}}\right) \right) &\leq &\ds \sum_{j=0}^{4m_k}\mu\left( f^{-j}\left(-\frac{1}{m_k^{3}}, \frac{1}{m_k^{3}}\right)\right)=\sum_{j=0}^{4m_k}\mu\left(-\frac{1}{m_k^{3}}, \frac{1}{m_k^{3}}\right) \nonumber\\
&\leq &2\,c\sum_{j=0}^{4m_k}\frac{1}{m_k^{3}}=2c\left(\frac{4m_k+1}{m_k^{3}}\right).
\end{eqnarray}

In what follows, given $A\subset \RR,\,\,\# A$ 
denotes the cardinality of $A$. 
\begin{claim} \label{claim1}For any $k$ and any $1\leq i\leq r_k+1$ there is a set $\mathcal{R}_i\subset \{1,\dots, 2^{m_k}\}$ with $\#\mathcal{R}_{i}=2^{m_k-1}$, such that for each $r\in \cR_i$ there is a point $x\in J^{(k)}_{i,r}$ such that  
$$x\notin \ds \bigcup_{j=1}^{4m_k}f^{-j}\left(-\frac{1}{m_k^{3}}, \frac{1}{m_k^{3}}\right).$$ 
\end{claim}

\begin{proof}The idea of the proof is to count the number of intervals that does not satisfy this property. 
 To do that, consider the set
 $$\cR_i^{C}:=\left\lbrace j \ : J_{i,j}^{(k)}\subset \ds \bigcup_{j=0}^{4m_k}f^{-j}\left(-\frac{1}{m_k^{3}}, \frac{1}{m_k^{3}}\right)\right\rbrace. 
 $$ 
We want  show that 
$\#\cR_i^{C}<2^{m_k-1}$. 
For this we proceed as follows.
Put $\#\cR_i^{C}=2^{m_k-n_k}+N_k$ with $0\leq N_k< 2^{m_k-n_k}$ and {let $\underline{j}\in \{1,\dots, 2^{m_k}\}$. Then, by the definition of $J_{i,{j}}^{(k)}$, we obtain $\mu(J_{i,\underline{j}}^{(k)})= \mu(J_{i,{j}}^{(k)})$ for all $j\in\cR_{i}^{C}$. } 
Hence, equations (\ref{E1L3}) and (\ref{E2L3}) imply that 
$$\frac{1}{2^{m_k}(m_k+1)}(2^{m_k-n_k}+N_k) <\mu(J_{i,\underline{j}}^{(k)})\cdot \#\cR_i^{C}\leq\mu\left( \ds\bigcup_{j\in \cR_i^{C}}J_{i,j}^{(k)}\right) \leq 2c\left(\frac{4m_k+1}{m_k^{3}}\right).$$
Hence we have 
$$\frac{1}{2^{n_k}}\leq \frac{1}{2^{n_k}}+\frac{N_k}{2^{m_k}}\leq 2c\cdot\frac{(4m_k+1)(m_k+1)}{m_k^{3}}\leq \frac{20c}{{m_k}}
$$
which implies that if $m_k$ is large enough ($m_k> 40 c$), then $n_k$ should be bigger than $1$, i.e., $n_k>1$.

Now $N_k < 2^{m_k-n_k}$ implies that $\frac{N_k}{2^{m_k}} < \frac{1}{2^{n_k}}$ and as $n_k > 1$, we get 
$$\frac{1}{2^{n_k}}+\frac{N_k}{2^{m_k}}<\frac{2}{2^{n_k}}\leq \frac{1}{2}.$$
Thus $\#\cR_i^{C}=2^{m_k-n_k}+N_k<2^{m_k-1}$ and this concludes the proof of claim \ref{claim1}.
\end{proof}
\begin{claim}\label{1'}
Consider the set $\cR_{i}^{+}=\{r\in \cR_i:|J_{i,r}^{(k)}|\geq\frac{1}{3m_k^3}\}$. Then  $\#\cR_{i}^{+}<2^{m_k-2}$.
\begin{proof}
As the intervals $J_{i,j}^{(k)}$ are pairwise disjoints, if $\#\cR_{i}^{+}\geq 2^{m_k-2}$ then 
$$1\geq \vert\bigcup_{j\in \cR_{i}^{+}}J_{i,j}^{(k)}\vert\geq \frac{2^{m_k-1}}{3m_k^3},$$
which implies a contradiction {for} $m_k$  large enough.
\end{proof}
\end{claim}
The above claim ensure that the set $\widetilde{\cR}_{i}:=\cR_{i}\setminus \cR_{i}^{+}$ has cardinality $\#\widetilde{\cR}_{i}\geq 2^{m_k-2}$.
\begin{claim}\label{claim2}  For all $r\in \widetilde{\cR}_i$  there is $j(i,r) \, {\in \{1,\dots, 4m_k\}}$ minimal, such that 
\begin{equation}\label{E3L3}
|f^{j(i,r)}(J_{i,r}^{(k)})|>\frac{1}{3m_k^{3}}.
\end{equation}
\end{claim}
\begin{proof} 
\indent Let  $r\in \widetilde{\cR}_{i}$, then if $|f^{s}(J_{i,r}^{(k)})|\geq\frac{1}{m_k^{3}}>\frac{1}{3m_k^{3}}$ for some $s\in j=1,\dots, 4m_k$, we are done.
Otherwise, assume that there is $s\in \{1,\dots,4m_k\}$ such that $|f^{t}(J_{i,r}^{(k)})|<\frac{1}{m_k^{3}}$ for all $1\leq t\leq s$.
 If $0\in f^{t_0}(J_{i,r}^{(k)})$ for some  $1\leq t_0\leq s$,   claim \ref{claim1} {implies that} there is $x_r\in J_{i,r}^{(k)}$ such that $x_r\notin f^{-j}\left(-\frac{1}{m_k^{3}}, \frac{1}{m_k^{3}}\right)$ for $j=1,\dots, 4m_k$, {and so} we get  $|f^{t_0}(J_{i,r}^{(k)})|>\frac{1}{m_k^3}$, contradicting our hypothesis.
Thus $0\notin f^{t}(J_{i,r}^{(k)})$ for all $1\leq t\leq s$. Since $f^{s}$ acts as a diffeomorphism on $J_{i,r}^{(k)}$ with derivative $|(f^{s})'|>\eta^s>2^{s/2}$ { and equation (\ref{E1L3}) holds}, {we obtain}  
 $$
\frac{1}{m_k^{3}} \geq |f^s(J_{i,j}^{(k)})| \geq \frac{\eta^s}{2^{m_k}(m_k+1)}\geq\frac{2^{s/2}}{2^{m_k}(m_k+1)}=
 \frac{2^{s/2 - m_k}}{m_k + 1}  \Longrightarrow s/2-m_k < 0 
 \Longrightarrow s< 2m_k.
 $$
If $|f^{s+1}(J_{i,r}^{(k)})|>\frac{1}{3m_k^{3}}$, then we are done. Otherwise, if $|f^{s+1}(J_{i,r}^{(k)})|\leq\frac{1}{3m_k^{3}}<\frac{1}{m_k^{3}}$,  reasoning as before, we get that $0\notin f^{s+1}(J_{i,r}^{(k)})$. Since $0\notin f^{s+1}(J_{i,r}^{(k)})$, then $f$ acts as a diffeomorphism on $f^{s}(J_{i,r}^{(k)})$ with derivative $|f'|>\eta$, which allows to state that $|f^{s+1}(J_{i,r}^{(k)})|>\eta|f^{s}(J_{i,r}^{(k)})|$. Again, if $|f^{s+2}(J_{i,r}^{(k)})|>\frac{1}{3m_k^{3}}$, we are done. Otherwise, if $|f^{s+2}(J_{i,r}^{(k)})|\leq\frac{1}{3m_k^{3}}<\frac{1}{m_k^{3}}$, and reasoning as before, we get that $0\notin f^{s+2}(J_{i,r}^{(k)})$ and then 
$$|f^{s+2}(J_{i,r}^{(k)})|>\eta|f^{s+1}(J_{i,r}^{(k)})|>\eta^2|f^{s}(J_{i,r}^{(k)})|.$$
Using this argument recursively, if  $|f^{s+2m_k-1}(J_{i,r}^{(k)})|\leq\frac{1}{3m_k^{3}}<\frac{1}{m_k^{3}}$, then $0\notin f^{s+2m_k-1}(J_{i,r}^{(k)})$ and it holds 
$$|f^{s+2m_k}(J_{i,r}^{(k)})|>\eta|f^{s+2m_k-1}(J_{i,r}^{(k)})|>\cdots>\eta^{2m_k}|f^{s}(J_{i,r}^{(k)})| \ \
\text{and} \ \ |f^{s}(J_{i,r}^{(k)})|>\eta^{s}|J_{i,r}^{(k)}|.$$
Thanks to inequality (\ref{E1L3}) {we  conclude} that   
$$|f^{s+2m_k}(J_{i,r}^{(k)})|>\frac{2^{m_k}\eta^s}{2^{m_k}(m_k+1)}=\frac{\eta^s}{m_k+1}>\frac{1}{3m_k^3},$$
{finishing the proof of claim \ref{claim2}.}
 \end{proof}
 Now consider the sequence of intervals $f^{j(i,r)}(J_{i,r}^{(k)})$ given by claim \ref{claim2}.\\ Since $|f^{j(i,r)}(J_{i,r}^{(k)})|>\frac{1}{3m_k^3}$ for all $r\in \widetilde{\cR}_i$, we can apply lemma \ref{L2-ALEO} and corollary \ref{C1-ALEO} to get the following: 

\begin{claim}\label{claim3}
For all $r\in \widetilde{\cR}_{i}$, there is an interval $I_{i,r}^{(k)}\subset f^{j(i,r)}(J_{i,r}^{(k)})$ and integer $m_{i,r}^{(k)}$ such that  $f^{m_{i,r}^{k}}\colon I_{i,r}^{(k)} \to L_1^{a}$ is a diffeomorphism, 
$0\notin f^{s}(I_{i,r}^{(k)})$ for $s=0,1, \dots, m_{i,r}^{(k)}-1$, \newline $m_{i,r}^{(k)}\leq D\log m_k$ and $\ds\sup_{x\in f^{s}(I_{i,r}^{(k)})}|f'|=E\cdot m_k^{\xi}$.
\end{claim}

\begin{claim}\label{claim4}Let $\tilde{I}_{i,r}^{(k)}\subset J_{i,r}^{(k)}$ with $f^{j(i,r)}(\tilde{I}_{i,r}^{(k)})=I_{i,r}^{(k)}$, where $I_{i,r}^{(k)}$ is as in $claim \, \ref{claim3}$.
Then, there is a constant $H>0$, depending only of $f$, such that 
\begin{equation}\label{E5L3}
|\tilde{I}_{i,r}^{k}|\geq H|{I}_{i,r}||J_{i,r}^{(k)}|.
\end{equation}
\end{claim}
\begin{proof} First note that the Mean Value Theorem implies
\begin{equation}\label{E9'}
\frac{|\tilde{I}_{i,r}^{(k)}|}{|J_{i,r}^{(k)}|}=\frac{|(f^{j(i,r)})'(y)|}{|(f^{j(i,r)})'(x)|}\cdot\frac{|{I}_{i,r}|}{|f^{j(i,r)}(J_{i,r}^{(k)})|} \ \ \text{for some} \ \ x\in \tilde{I}_{i,r}^{(k)} \, ; \,y\in J_{i,r}^{(k)}.
\end{equation}

It is enough to bound $\ds\frac{|(f^{j(i,r)})'(y)|}{|(f^{j(i,r)})'(x)|}$, {since equality (\ref{E9'}) implies that inequality (\ref{E5L3}) holds.}
For this sake, we proceed as follows.
As  $j(i,r)$ is minimal satisfying  (\ref{E3L3}) we get 
\begin{equation}\label{E5'L3}
|f^{s}(J_{i,r}^{(k)})|<\frac{1}{3m_k^{3}} \ \ \text{for} \ \ s=0,\dots, j(i,r)-1.
\end{equation}
\noindent This implies, reasoning as in the proof of claim \ref{claim2}, that  $0\notin f^{s}(J_{i,r}^{(k)})$ for $s=0,\dots, j(i,r)-1$ and hence $f^{s}|_{J_{i,r}^{(k)}}$ is a diffeomorphism for $s=0,\dots, j(i,r)-1$. 

Observe that by  claim \ref{claim1}, for each $s\in \{0,\dots, j(i,r)-1\}$, there is $x_s\in J_{i,r}^{(k)}$ such that $f^{s}(x_s)\notin (-\frac{1}{m_k^{3}},\frac{1}{m_k^{3}})$, and so, if 
 $d(\cdot,\cdot)$ is the distance between sets, by equation (\ref{E5'L3}) we conclude that 
\begin{equation}\label{E5''L3}
\inf_{x\in {J_{i,r}^{(k)}}}|f^{s}(x)|=:d(f^{s}({J_{i,r}^{(k)}}),\{0\})>\frac{1}{2m_k^{3}}.
\end{equation}
Now we have
\begin{eqnarray}\label{E5'''L3}
\Bigg|\log\frac{(f^{j(i,r)})'(y)}{(f^{j(i,r)})'(x)}\Bigg| &=& \Bigg|\ds \sum_{s=0}^{j(i,r)-1}\log(f'(f^s(y))-\log(f'(f^s(x))\Bigg|\nonumber \\
&\leq & \sum_{s=0}^{j(i,r)-1}\vert\log(f'(f^s(y))-\log(f'(f^s(x))\vert \nonumber\\
&\leq_{\tiny{by \, \, MVT}} & \sum_{s=0}^{j(i,r)-1} \frac{|f''(f^s(z_s)|}{|f'(f^s(z_s))|} |f^s(y) - f^s(x)|,  \,\,\,                \text{for some} \ \ z_s\in J_{i,r}^{(k)}, \nonumber\\
&\leq_{\tiny{by \, \, (\ref{E3S2})}}&\sum_{s=0}^{j(i,r)-1}C \cdot C_1 \cdot \frac{1}{| f^s(z_s)|} \cdot  | f^s(J_{i,r}^{(k)})|, \text{ where $C, C_1$ depend only on $f$} \nonumber\\
&\leq_{\tiny{by \, \, (\ref{E5''L3})}} &\sum_{s=0}^{j(i,r)-1} C \cdot  C_1 \cdot 2m_k^3\cdot | f^s(J_{i,r}^{(k)})|. 
\end{eqnarray}
Recall that equation (\ref{E3S2}) implies that {$\frac{\vert f''(x)\vert}{\vert f'(x)\vert}\leq\frac{ C\cdot C_1}{|x|}$}, with $C_1$, $C$ depending only  of $f$.
Thus, since $f^{s}|_{J_{i,r}^{(k)}}$ is a diffeomorphim for each $s\in \{0,\dots, j(i,r)-1\}$, and satisfies  Property (f2) (see subsection \ref{ProOneMap}), we get
$$|f^{j(i,r)-1}(J_{i,r}^{(k)})|\geq\sqrt{2}|f^{j(i,r)-2}(J_{i,r}^{(k)})|\geq\sqrt{2}^{2}|f^{j(i,r)-3}(J_{i,r}^{(k)})|\geq\cdots>\sqrt{2}^{s}|f^{j(i,r)-(s+1)}(J_{i,r}^{(k)})|.$$
{Making} the change of variable $t=j(i,r)-(s+1)$,  the last inequality provides 
\begin{equation}\label{E5''''L3}
|f^{j(i,r)-1}(J_{i,r}^{(k)})|\geq (\sqrt{2})^{j(i,r)-t-1}|f^{t}(J_{i,r}^{(k)})|.
\end{equation}
Using the inequality (\ref{E5''''L3}) together with  (\ref{E5'L3}) and replacing in the last term of equation (\ref{E5'''L3}) we get that 
\begin{equation}\label{e1-claim4}
\Bigg|\log\frac{(f^{j(i,r)})'(y)}{(f^{j(i,r)})'(x)}\Bigg|\leq C\cdot C_1\sum_{t=0}^{j(i,r)-1}2m_{k}^{3}\cdot \frac{1}{3m_{k}^3}
\left( \frac{1}{\sqrt{2}}\right)^{j(i,r)-t-1}< \frac{2}{3}\cdot C\cdot C_1\cdot\sqrt{2}(\sqrt{2}+1).
\end{equation}
Setting $H:=e^{-\frac{2}{3}\cdot C\cdot C_1\cdot \sqrt{2}\cdot (\sqrt{2}+1)}$, we bound 
$\ds\frac{|(f^{j(i,r)})'(y)|}{|(f^{j(i,r)})'(x)|}$ and  inequality (\ref{E9'}) follows,  implying that inequality (\ref{E5L3}) holds. The proof of claim \ref{claim4} is finished.
\end{proof} 

The next step is to construct the regular Cantor with Hausdorff dimension close to $1$.
For this sake, we consider the collection of surjective maps 
$$\{g_{i,r}=f^{m_{i,r}^{(k)}}\circ f^{j(i,r)}\colon \tilde{I}_{i,r}^{(k)}\to L_1^{a}\,\, | \,\, r\in \widetilde{\cR}_{i}\}.$$
Let $\ds g_{ki}:L_k^{i}=\bigcup_{r\in\widetilde{\cR}_{i}}\tilde{I}_{i,r}^{(k)}\to L_1^a$ be defined by $g_{ki}=g_{i,r}|_{\tilde{I}_{i,r}^{(k)}}$ 
and $C_k^ i$ be the regular Cantor set defined by the intervals $\tilde{I}_{i,r}^{(k)}$ and $g_{i,r}$, i.e. 
\begin{equation*}
C_k^{i}=\bigcap_{n\geq1}g_{ki}^{-n}(L_k^{i}).
\end{equation*}
The final step is to show that $HD(C_k^{i})\to 1$ as $k\to+\infty$. For this, we {use the same strategy given in \cite[Theorem 3]{PT}.  In fact, }consider the number 
$$\Lambda_{1,\tilde{I}_{i,r}^{(k)}}=\ds\sup_{x\in \tilde{I}_{i,r}^{(k)}}|g'_{i,r}|$$ and define $d_1\in [0,1]$ by  
\begin{equation}\label{eq2-claim4}
\sum_{r\in \widetilde{\cR}_{i}}(\Lambda_{1,\tilde{I}_{i,r}^{(k)}})^{-d_1}=1.
\end{equation}

It is shown in \cite[pp 69-70]{PT} that $d_1\leq HD(C_k^{i})$. Therefore, we can estimate  $HD(C_k^{i})$ by computing $d_1$.\\
{To do that, note that } $g_{i,r}=f^{m_{i,r}^{(k)}}\circ f^{j(i,r)}$, { and to simplify notations,}  denote $h_1=f^{m_{i,r}^{(k)}}$ and $h_2=f^{j(i,r)}$. Then 
$$\ds\Lambda_{1,\tilde{I}_{i,r}^{(k)}}\leq \sup_{I_{i,r}^{(k)}}|h_1'|\cdot \sup_{\tilde{I}_{i,r}^{(k)}}|h_2'|.$$
Corollary \ref{C1-ALEO} {gives} that 
$\ds\sup_{I_{i,r}^{(k)}}|h_1'|\leq E^{D\log m_k}\cdot m_k^{F\log m_k}$, where $F=D\cdot \xi$. 
To estimate the {the supremum of $|h_2'|$, $\sup |h_2'|$, in
$\tilde{I}_{i,r}$}, we note that by  the proof of claim \ref{claim4} the function  $h_2|_{\tilde{I}_{i,r}^{(k)}}$ has bounded distortion.
 Thus 
$$\sup_{\tilde{I}_{i,r}^{(k)}}|h_2'|\leq H^{-1}\inf_{\tilde{I}_{i,r}^{(k)}}|h_2'|,$$
where $H:=e^{-\frac{2}{3}\cdot C\cdot C_1\cdot \sqrt{2}\cdot (\sqrt{2}+1)}$ (see equation (\ref{e1-claim4})).
Since $h_2(\tilde{I}_{i,r}^{(k)})={I}_{i,r}^{(k)}$, {the mean value theorem implies}
\begin{equation*}\label{eq1-final}
\inf_{\tilde{I}_{i,r}^{(k)}}|h_2'|\leq \frac{|{I}_{i,r}^{(k)}|}{|\tilde{I}_{i,r}^{(k)}|}\overset{\mathrm{\, by \, (\ref{E5L3})}}{\leq} \frac{H^{-1}}{|J_{i,r}^{(k)}|}\overset{\mathrm{\, by \, (\ref{E1L3})}}{\leq} H^{-1}2^{m_k}(m_k+1).
\end{equation*}
The last two {inequalities  imply} that 
\begin{equation}\label{eq3-claim4}
\ds\Lambda_{1,\tilde{I}_{i,r}^{(k)}}\leq H^{-2}\cdot E^{D\log m_k}\cdot m_k^{F\log m_k}\cdot 2^{m_k}(m_k+1)=2^{(1+o(1))m_k},
\end{equation}
since $\ds \lim_{k\to \infty}\dfrac{\log H^{-2}+ D\log m_k\cdot \log E+ F\cdot(\log m_k)^2+ m_k \log 2+ \log (m_k+1)}{m_k}=\log 2$.\\
Therefore, since $\# \widetilde{\cR}_{i}\geq 2^{m_k-2}$,  inequalities  (\ref{eq2-claim4}) and (\ref{eq3-claim4}) imply that  
$$2^{m_k-2}\cdot \left(\frac{1}{2^{(1+o(1)) m_k}}\right)^{d_1}\leq 1=\sum_{r\in \widetilde{\cR}_{i}}(\Lambda_{1,\tilde{I}_{i,r}^{(k)}})^{-d_1}.$$ 
Hence  
\begin{eqnarray*}
(m_k-2)\log 2&\leq& (1+o(1))\cdot m_k\cdot d_1\cdot\log 2   \Longrightarrow\\
1-o(1)&=&\frac{m_k-2}{m_k}\leq (1+o(1))\cdot d_1 \Longrightarrow\\
1-o(1)&\leq&(1+o(1))\cdot d_1. 
\end{eqnarray*}
Thus, $1-o(1)\leq d_1\leq 1$. Now we define $C_k:=\cup_{i}C_k^i$ which satifies the condition of the theorem,   finishing the proof of theorem \ref{L.Principal}.
\end{proof}

As an immediate consequence of theorem \ref{L.Principal}, we have the following 
\begin{maincorollary}\label{Dim. of bi-dim attractor}
The Hausdorff dimension of the bi-dimensional attractor for the Poincar\'e map $P$, $\Lambda_{P}$,  is strictly greater than $1$. 
\end{maincorollary}
\begin{proof}
For this, let $\Gamma=\{(x,y,1):x=0\}$ and
$$\Lambda_{P}=\overline{\bigcap_{i\geq 1}P^{i}(S\setminus \Gamma)},\,\,
\mbox{ be as in equation (\ref{atrator})}.
 $$

\noindent For each $k> 0$, let $C_k$ be the regular Cantor set given by theorem \ref{L.Principal} and define 
\begin{equation}\label{E1SBA}
\Lambda_{P}^{k}=\{(x,y)\in \Lambda_P\colon x\in C_k\}.
\end{equation}

\noindent Notice that by construction, each $C_k$ is a regular Cantor set (see comments after definition \ref{Regular Cantor set}) and so, for each $k$, $\Lambda_{P}^{k}$ is a basic set for $P$. Moreover, $\Lambda_{P}^{k}\subset \Lambda_{P}^{k+1}$ and by proposition \ref{Prop.2}
$$HD(\Lambda_{P}^{k})=HD(_uK_P^{k})+HD(_sK_P^{k})=HD(C_k)+HD(_sK_P^{k}),$$
where $_sK_P^{k}$ and $_uK_P^{k}=C_k$ are the stable and unstable Cantor sets associated to the basic set $\Lambda_{P}^{k}.$  
As $_sK_P^{k}$ is a regular Cantor set,  by proposition \ref{Prop.3}, there is $\xi>0$ such that $HD(_sK_P^{1})>\xi$. Hence $$HD(\Lambda_{P}^{k})=HD(C_k)+HD(_sK_P^{k})\geq HD(C_k)+HD(_sK_P^{1})>HD(C_k)+\xi.$$
Thus, theorem \ref{L.Principal} implies that  $H(\Lambda_{P}^{k})>1$ for $k$ large enough.  Since $\Lambda_{P}^{k}\subset \Lambda_{P}$, this finishes  the proof of Corollary \ref{Dim. of bi-dim attractor}. 
\end{proof}

\begin{proof}[\textbf{Proof of theorem \ref{T1}}]
Note that the geometric Lorenz attractor $\Lambda$ satisfies
$$\Lambda=\left( \bigcup_{t\in \re}X^{t}(\Lambda_{P})\right) \cup O,
\quad \mbox{where $O$ is the singularity.}
$$
Thus, 
$$HD(\Lambda)\geq 1+HD(\Lambda_P)>2.$$
The proof of theorem \ref{T1} is complete.
\end{proof}
We finished this section by announcing a corollary of the proof of  theorem \ref{L.Principal}  that might be of interest to the reader.
\begin{maincorollary}
If $f$ is a $C^2$ function that satisfies the properties $(f1)-(f3)$ {described in} section $\ref{ProOneMap}$ with  $f(-\frac{1}{2})\neq -\frac{1}{2}$, $f(\frac{1}{2})\neq \frac{1}{2}$
and also satisfies  equation (\ref{E3S2}), then there is an increasing family of regular Cantor sets $C_k$ for $f$ such that 
$$HD(C_k)\to 1 \ \ \text{as} \ \ k\to +\infty.$$
\end{maincorollary}
\section{Lagrange and Markov Spectra: proof of Theorem~ \ref{T2} }\label{sec-spectra}

In this section we prove  theorem \ref{T2}. For this, we first prove  that 
small perturbations of the Poincar\'e map $P$ restricted to $\Lambda_{P}^{k}$, with  $\Lambda_{P}^{k}$
defined at (\ref{E1SBA}), can be realized as Poincar\'e maps of small perturbations of the initial geometric Lorenz flow $X^t$ (lemma \ref{All-Pert}). 
{Then}, taking $k$ such that $HD(\Lambda_{P}^{k})>1$, we recover 
the properties described in \cite{RM} needed to 
apply \cite[Main Theorem]{RM}, obtaining non empty interior in the  Lagrange and Markov spectrum.

We start announcing the main theorem in \cite{RM}
which is a fundamental tool to obtain theorem \ref{T2}. 
Given $A\subset M$, $\interior(A)$ denotes the interior of $A$.
\vspace{0.2cm}

{
\noindent \textbf{Theorem}\,[Main Theorem at \cite{RM}]\emph{Let $\Lambda$ be a horseshoe associated to a $C^2$-diffeomorphism $\varphi$ such that  $HD(\Lambda)>1$. Then there is, arbitrarily close to $\varphi$, a diffeomorphism $\varphi_{0}$ and a $C^{2}$-neighborhood $\mathcal{W}$ of $\varphi_{0}$ such that, if $\Lambda_{\psi}$ denotes the continuation of $\Lambda$ associated to $\psi\in \mathcal{W}$, there is an open and dense set $H_{1}(\psi, \Lambda_{\psi})\subset C^{1}(M,\re)$ such that for all $f\in H_{1}(\psi, \Lambda_{\psi})$, we have 
\begin{equation*}
\interior( \ L(\psi, \Lambda_{\psi}, f ))\neq\emptyset \quad \text{and}
\quad   
\interior( M(\psi, \Lambda_{\psi}, f)) \neq\emptyset.
\end{equation*}
}}

The set $H_{1}(\psi, \Lambda_\psi)$ is described by 
\begin{equation*}\label{E-Descrip}
{H}_{\psi}=\left\{f\in C^{1}(M,\re):\#M_{f}(\Lambda_{\psi})=1, \  z\in M_{f}(\Lambda_{\psi}), \ D\psi_{z}(e_{z}^{s,u})\neq 0\right\},
\end{equation*}
 where $M_{f}(\Lambda_{\psi}):=\{z\in \Delta: f(z)\geq f(x)\ \text{for all} \, x\in \Lambda_{\psi}\}$ is  the set of maximum points of $f$ in $\Lambda_{\psi}$ and $e_{z}^{s,u}$ are unit vectors in $E^{s,u}(z)$,  respectively.  

\subsection{Perturbations of Poincar\'e Map }\label{PPM}

Fixed $k$ with $HD(\Lambda_{P}^{k})>1$. By construction, there is $\epsilon>0$ small so that $d(\Lambda_{P}^{k},\Gamma)>2\epsilon$, where $\Gamma=\{(x,y,1):x=0\}$.
Let $\mathcal{U}_{P}$ be a $C^{2}$ neighborhood of $P$ such that, if $\tilde{P}\in \mathcal{U}_{P}$ and $\Lambda_{\tilde{P}}^{k}$ is the hyperbolic continuation of $\Lambda_{P}^{k}$, then $d(\Lambda_{\tilde{P}}^{k}, \Gamma)>\epsilon$.

The next lemma states that in a neighborhood of $\Lambda_{\tilde{P}}^{k}$, we can recover $\tilde{P}\in \mathcal{U}_{P}$ as a Poincar\'e map associated to a geometric Lorenz flow $\tilde{X}^t$, $C^2$-close to $X^t$.

\begin{Le}\label{All-Pert}
Given $\tilde{P}\in\mathcal{U}_{P}$ there is a geometric Lorenz flow $\tilde{X}^t$,  $C^2$-close to $X^t$, such that the 
restriction to $\Lambda_{\tilde{P}}^{k}$ of the Poincar\'e map associated to $\tilde{X}^t$ 
coincides with the restriction of $\tilde{P}$ to $\Lambda_{\tilde{P}}^{k}$.

\end{Le}
\begin{proof}
For the proof  we construct explicitly a flow $\tilde{X}^t$, with the desired properties.  For this, we proceed as follows.
  
\noindent Let $\tilde{\mathcal{R}}=R_1\cup R_2\cup \cdots\cup R_m$ be a Markov partition of $\Lambda_{\tilde{P}}^{k}$ and let $U_i\subset S$ be an open set with $R_i\subset U_i$, $d(U_i,\Gamma)>\frac{\epsilon}{2}$ for all $i$, and such that if $\tilde{P}(x,y)\in R_i$ then $P(x,y)\in U_i$.
 The tubular flow theorem applied to $X$,  give local charts 
 $\psi_{i}:U_i\times [-1,1]\to \mathbb{R}^{3}$ for $i \in \{1,\cdots, m\}$
 satisfying
\begin{equation}\label{E111P}
\psi_{i}(U_i\times \{0\})\subset S \ \ \ \text{and} \ \ \ D(\psi_{i})_{(x,y,t)}(0,0,1)=X(\psi_{i}(x,y,t)).
\end{equation}
Put $W_i:=\psi_{i}(U_i\times (-1,1))$.  Without loss of generality, we can assume that 
\begin{equation*}\label{E11P}
W_i\cap W_j=\emptyset \quad \mbox{if}\quad i\neq j.
\end{equation*}
 We denote by $\tilde{P}_{i}$ and  $P_{i}$ the map $\tilde{P}$ and $P$ in these coordinates. 

\noindent Let   $\varphi:\re\to\re$ be a $C^{\infty}$ bump function such that $\varphi(t)= 0 $ for $t\leq -1$ and $\varphi(t)=1$ for $t\geq 1$.
Define the following flow on $U_i\times [-1,1]$:

$$\phi^{t}_{i}(x,y,0)=(P_i(x,y)+\varphi(3t+1)(\tilde{P}_i(x,y)-P(x,y)),t).$$
Note that 
\begin{equation}\label{E1P}
\phi^{t}_{i}(x,y,0)=\left\{\begin{array}{ll}(P_i(x,y),t),\ \ \text{if}\ \ t\leq-\frac{2}{3}\\ (\tilde{P}_{i}(x,y),t),\ \ \text{if}\ \  t\geq 0\end{array}\right..
\end{equation}
Consider the vector field on $U_i\times [-1,1]$ given by 
\begin{equation}\label{E22P}
Z_{i}(\phi^{t}_{i}(x,y,0))=\frac{\partial}{\partial t}\phi^{t}_{i}(x,y,0)=(3\varphi'(3t+1)(\tilde{P}_i(x,y)-P(x,y)),1).
\end{equation}
By the equation (\ref{E1P}), this vector field satisfies 
\begin{equation}\label{E2P}
Z_{i}(\phi^{t}_{i}(x,y,0))=\left\{\begin{array}{ll}(0,0,1),\ \ \text{if}\ \ t\leq-\frac{2}{3}\\ (0,0,1),\ \ \text{if}\ \  t\geq 0\end{array}\right..
\end{equation}
Let $Y_{i}$ be the vector field on $W_i=\psi_{i}(U_i\times (-1,1))$ defined by 
\begin{equation*}\label{E3P}
Y_{i}(\psi_{i}(x,y,t))=D(\psi_{i})_{(x,y,t)}(Z_{i}(\phi^{t}_{i}(x,y,0))).
\end{equation*}
By equations (\ref{E111P}) and (\ref{E2P}) we get that 
\begin{equation}\label{E33P}
Y_{i}(\psi_{i}(x,y,t))=X(\psi_{i}(x,y,t)) \ \ \text{for} \ \ t\leq -\frac{2}{3} \ \ \text{and}\ \ t\geq 0.
\end{equation}
\noindent Let $\mathcal{W}$ be the open set $\mathcal{W}=\bigcup_{i=1}^{m}\psi_{i}(U_i\times (-1,1))=\bigcup_{i=1}^{m}W_i$ and consider
the vector field $Y: \mathcal{W} \to \mathbb{R}^{3}$ given by
$Y=Y_i|_{W_i}$. 
Finally, define the  vector field $\tilde{X}$ by
\begin{equation*}\label{E4P}
\tilde{X}:=\left\{\begin{array}{ll}Y,\ \ \text{on}\ \ \mathcal{W}\\ X,\ \ \text{outside of}\ \ \mathcal{W} \end{array}\right..
\end{equation*}
Since $\tilde{P}\in \mathcal{U}_P$,  equation (\ref{E22P}) implies that $\tilde{X}$ is $C^2$-close to $X$.
If $\tilde{X}^{t}$ is the flow associated to the vector field $\tilde{X}$, equations (\ref{E1P}) and (\ref{E33P}) imply that the Poincar\'e map associated to  $\tilde{Y}^{t}$ restricted to $\Lambda_{\tilde{P}}^{k}$ is equal to $\tilde{P}$ restricted to $\Lambda_{\tilde{P}}^{k}$. To finish the proof, note that $d(U_i,\Gamma)>\frac{\epsilon}{2}$ for all $i$, and thus, $\tilde{X}^{t}$ is a geometric Lorenz flow, as desired. 
\end{proof}
\subsection{Regaining the Spectrum}\label{RS}

Recall that we are interested in study the spectrum over a {geometric Lorenz attractor $\Lambda$, that is not a hyperbolic set, as well as $\Lambda\cap S$. Thus, we cannot apply directly the techniques developed in the hyperbolic setting to analyze the spectrum in this case.   
So, the strategy we adopt is to profit from the fact that $\Lambda\cap S$ contains hyperbolic sets $\Lambda_{P}^k$ for the Poincar\'e map $P$  with Hausdorff dimension bigger than $1$. 
 Then we  use  similar arguments developed in  \cite{RM} to show that the Lagrange and Markov Dynamical Spectrum has non-empty interior for a set of $C^1$ real functions over the cross section $S$ and with these functions regaining the spectrum over $\Lambda$.
 In this direction, we proceed as follows.

\indent The dynamical Lagrange and Markov spectra of $\Lambda$ and $\Lambda_{P}^{k}$ are related in the following way. 
Given a function $F\in C^{s}(U,\mathbb{R})$, $s\geq 1$, let us denote by $f =maxF_{\phi}\colon D_{P}\to\re$ the function
$$max F_{\phi}(x):=\max_{0\leq t \leq t_{+}(x)}F(\phi^{t}(x)),$$
where $D_{P}$ is the domain of $P$ and $t_{+}(x)$ is such that  $P(x)=X^{t_{+}(x)}(x)$ and $U$ a neighborhood of $\Lambda$ as  theorem \ref{T2}. 
\begin{R}
The map $f =maxF_{\phi}$ might be not  $C^1$ in general.
\end{R}
{For all $x\in \Lambda_P^k$ we have }
$$\limsup_{n\to +\infty}f(P^n(x))=\limsup_{t\to + \infty}F(X^t(x))
\quad \mbox{and}\quad 
\sup_{n\in \mathbb{Z}}f(P^n(x))=\sup_{t\in \re}F(X^t(x)).$$
{In particular, if  $\Lambda^k=\bigcup_{t\in\re}X^{t}(\Lambda_P^k)\subset \Lambda$ we get}
\begin{equation*}\label{Continuos and Discrete}
L(X, \Lambda^k, F)=L(P, \Lambda_P^k, f) \ \ \text{and} \ \ M(X, \Lambda^k, F)=M(P, \Lambda_P^k, f).
\end{equation*}
\begin{R}
It is worth to note that: given a vector field $Y$ close to $X$, then the flow of $Y$ still defines a Poincar\'e
map $P_Y$ defined in the same cross-sections where $P$ is defined. 
\end{R}
Thus, the last equality reduces theorem \ref{T2} to the following statement:
\begin{T}\label{Theorem 2}In the setting of theorem \ref{T2}, arbitrarily close to $X$ there is an open set $\mathcal{W}$ of $C^2$-vector fields {defined} on $U$ such that for 
{every} $Y\in \mathcal{W}$ { there is a
$C^{2}$  open and dense   subset}  $\mathcal{H}_{Y,\Lambda}\subset C^{2}(U,\re)$, such that 
$$\text{int}\, M(P_{Y},\Lambda_{P_Y}^{k}, max F_{Y})\neq \emptyset \, \, \text{and}\, \,  \, \text{int}\, L(P_{Y},\Lambda_{P_Y}^{k}, max F_{Y})\neq \emptyset$$
whenever $F\in \mathcal{H}_{Y,\Lambda}$.
 Here $\Lambda_{P_Y}^{k}$ denotes the hyperbolic continuation of  $\Lambda_{P}^{k}$.
\end{T}

\subsubsection{Description of ${\mathcal{H}}_{Y,\Lambda}$}\label{Descript of function} 
Given a compact hyperbolic set $\Delta$ for $P$ and a Markov partition $R$ of $\Delta$, we define the set 
\begin{equation}\label{E-Descrip}
{H}_{1}(P,\Delta)=\left\{f\in C^{1}(S\cap R,\re):\#M_{f}(\Delta)=1, \  z\in M_{f}(\Delta), \ DP_{z}(e_{z}^{s,u})\neq 0\right\},
\end{equation}
where { $S$ is the cross section as the section \ref{CGM}}, $M_{f}(\Delta):=\{z\in \Delta: f(z)\geq f(x)\ \text{for all} \ x\in \Delta\}$, the set of maximum points of $f$ in $\Delta$ and $e_{z}^{s,u}$ are unit vectors in $E^{s,u}(z)$  respectively (cf. \cite[section 3]{RM}).\\

\begin{Defi}
We say that  $F\in\mathcal{H}_{Y,\Lambda}\subset C^2(U,\re)$ if
there is a neighborhood $R_F$ of $\Lambda_{P_Y}^k$ such that 
\begin{itemize}
\item[(i)] 
$maxF_{Y}|_{S\cap R_F}\in C^1(S\cap R_F,\re).$
\item[(ii)] $maxF_{Y}\in H_1(P_Y,\Lambda_{P_Y}^k)\subset C^{1}(S\cap R_F,\re)$. 
\end{itemize}

\end{Defi}
With similar arguments of section 4 at \cite{RM2} we {prove the following result} 
\begin{Le}\label{set of functions}
The set ${\mathcal{H}}_{Y,\Lambda}$ is a dense  $C^2$-open set . 
\end{Le}
\begin{R}
If $Y$ is $C^2$ close enough of $X$, then $HD(\Lambda_{P_Y}^k)>1$\, $(cf. \,\, \cite[sec. 4.3]{PT})$.
\end{R}

\begin{proof}[\bf Proof of  Theorem \ref{Theorem 2}]
As $\Lambda_{P}^k>1$, by main theorem at \cite{RM}, arbitrarily close to $P$ there exists a $C^2$ open set $\widetilde{\mathcal{W}}$, such that for $\tilde{P}\in \widetilde{ \mathcal{W}}$ it holds 
$$\text{int}\, M(\tilde{P}, \Lambda_{\tilde{P}}^{k}, f )\neq \emptyset \ \ \text{and} \ \\, \text{int}\, M(\tilde{P}, \Lambda_{\tilde{P}}^{k}, f )\neq \emptyset,$$
whenever $f\in H_{1}(\tilde{P}, {\Lambda_{\tilde{P}}}^{k})$.
Note also that  lemma \ref{All-Pert} provides  a  neighborhood $\mathcal{W}$, $C^2$ close to $X$, such that for any $\tilde{P}\in \widetilde{ \mathcal{W}}$ there is $Y\in \mathcal{W}$  such that $P_{Y}=\tilde{P}$ in a neighborhood of $\Lambda_{\tilde{P}}^{k}$. Thus, for $Y\in { \mathcal{W}}$ and $F\in \mathcal{H}_{Y,P}$  it holds 
$$\text{int}\, M(P_{Y}, \Lambda_{P_{Y}}^{k}, max F_{Y}|_{S\cap R_F})\neq \emptyset \ \ \text{and} \ \ \text{int}\, L(P_{Y}, \Lambda_{P_{Y}}^{k}, max F_{Y}|_{S\cap R_F})\neq \emptyset,$$
since $max F_{Y}|_{S\cap R_F}\in H_{1}(P_Y, \Lambda_{P_{Y}}^{k})$. This finishes  the proof of theorem  \ref{Theorem 2} and so  concluding the proof of theorem \ref{T2}.

\end{proof}

\bibliographystyle{alpha}

\noindent Carlos Gustavo T. Moreira\\
Instituto de Matem\'atica Pura e Aplicada (IMPA),\\
Estrada Dona Castorina, 110,\\
 22460-320,  Rio de Janeiro - RJ.\\ 
e-mail {\em{gugu@impa.br}}\\

\noindent Maria Jos\'e Pacifico \,\,$\&$\,\, Sergio Roma\~na Ibarra\\
Instituto de Matem\'atica,
Universidade Federal do Rio de Janeiro.\\
C. P. 68.530, CEP 21.945-970, Rio de Janeiro, RJ.\\
e-mail: {\em{pacifico@im.ufrj.br \quad 
sergiori@im.ufrj.br}}

\end{document}